\documentclass[12pt, reqno]{amsart}
\usepackage{mathrsfs}
\usepackage{amssymb,amsthm,amsmath}
\usepackage[numbers,sort&compress]{natbib}
\usepackage{amssymb,amsmath}
\usepackage{amsfonts}
\usepackage{mathrsfs}
\usepackage{latexsym}
\usepackage{amssymb}
\usepackage{amsthm}
\usepackage{color}
\usepackage{pdfsync}
\usepackage{indentfirst}
\date{today}

\usepackage{amsmath}
\usepackage{amsthm}
   
   \newtheorem{remark}{Remark}

   		\newtheorem{theorem}{Theorem}[section]
      		\newtheorem{proposition}{Proposition}[section]
   		    \newtheorem{lemma}{Lemma}[section]

\newcommand{\beq}{\begin{equation}}
\newcommand{\eeq}{\end{equation}}
\newcommand{\ben}{\begin{eqnarray}}
\newcommand{\een}{\end{eqnarray}}
\newcommand{\beno}{\begin{eqnarray*}}
\newcommand{\eeno}{\end{eqnarray*}}

\numberwithin{equation}{section}

\begin{document}
\title[Suppression of blow-up in PKS-NS system via Poiseuille flow]{Suppression of blow-up in
			multi-species Patlak-Keller-Segel-Navier-Stokes system via the Poiseuille flow in a finite channel}
\author{Shikun~Cui}
\address[Shikun~Cui]{School of Mathematical Sciences, Dalian University of Technology, Dalian, 116024,  China}
\email{cskmath@163.com}

\author{Wendong~Wang}
\address[Wendong~Wang]{School of Mathematical Sciences, Dalian University of Technology, Dalian, 116024,  China}
\email{wendong@dlut.edu.cn}
\date{\today}
\maketitle

\begin{abstract}
In this paper, we consider the multi-species parabolic-elliptic Patlak-Keller-Segel system coupled with the Navier-Stokes equations
				near the 2-D Poiseuille flow $(\ A(1-y^2), 0\ )$
				in a finite channel $\Omega=\mathbb{T}\times\mathbb{I}$ with $ \mathbb{I}=(-1,1)$. Furthermore, the Navier-slip boundary condition is imposed on the perturbation of velocity $u$.
				We show that if the Poiseuille flow is sufficiently strong ($A$ is large enough),
				the solutions to the system are global in time without any smallness restriction
				on the initial cell mass.
\end{abstract}

\tableofcontents

		\section{Introduction}
		Consider the following multi-species  parabolic-elliptic Patlak-Keller-Segel (PKS) system coupled with the
		Navier-Stokes (NS) equations:
		\begin{equation}\label{ini}
			\left\{
			\begin{array}{lr}
				\partial_tn_1+v\cdot\nabla n_1=\triangle n_1-\chi_1\nabla\cdot(n_1\nabla c), \\
				\partial_tn_2+v\cdot\nabla n_2=\triangle n_2-\chi_2\nabla\cdot(n_2\nabla c), \\
				\triangle c+n_1+n_2-c=0, \\
				\partial_tv+v\cdot\nabla v+\nabla P=\triangle v+(n_1+n_2)\nabla\phi, \\
				\nabla\cdot v=0, \quad (x,y)\in\mathbb{T}\times\mathbb{I}, ~ \mathbb{I}=(-1,1)
			\end{array}
			\right.
		\end{equation}
		along with initial conditions
		$$(n_1,n_2,c,v)\big|_{t=0}=(n_{1,{\rm in}},n_{2,{\rm in}},c_{\rm in},v_{\rm in}),$$
		where  $c$ denotes the chemoattractant density,
		$n_1$ and $n_2$ denote the cell density of the first and second species, respectively.
		The positive constant $\chi_1$ and $\chi_2$
		represent the sensitivities of two species to the chemoattractant, and $v$ denotes the velocity of fluids.
		In addition, $P$ is the pressure and $\phi$ is a given potential function. We assume $\phi=y$ as Zeng-Zhang-Zi in \cite{zeng}.
		
		If $v=0$ and $\phi=0$, the system (\ref{ini}) becomes the classical multi-species Patlak-Keller-Segel system.
		In this system, the cooperative interaction between two species and their response to the chemoattractant are described.
		The single-species Patlak-Keller-Segel (PKS) system is a model in mathematical biology,
		primarily used to describe the chemotaxis of biological cells, i.e.
		how cells react to and move in response to changes in the concentration
		of chemical substances in their surrounding environment.
		The PKS system was jointly developed by Patlak \cite{Patlak1}, and Keller-Segel \cite{Keller1}.
		
		In recent years, Patlak-Keller-Segel system and its variants have received great attention in mathematics.
		The 2-D parabolic-parabolic PKS model has a critical mass of $8\pi$.
		When the cell mass $M:=||n_{\rm in}||_{L^1}$ is less than $8\pi$, Calvez-Corrias  \cite{Calvez1} proved the solutions of the system are global.
		When the cell mass is greater than $8\pi$, the solution will blow up in finite time obtained by Schweyer in \cite{Schweyer1}.
		In the case of the 2D parabolic-elliptic Patlak-Keller-Segel equation, it is globally well-posed if and only if the total mass $M\leq8\pi$ by Wei in \cite{wei11}.
		Generally, consider a PKS equation with the chemoattractant $\chi$:
		\begin{equation}\nonumber
			\left\{
			\begin{array}{lr}
				\partial_tn=\triangle n-\chi\nabla\cdot(n\nabla c), \\
				\triangle c+n=0, \\
				n(x,0)=n_0, \quad x\in\mathbb{R}^2.
			\end{array}
			\right.
		\end{equation}
		When $\chi M<8\pi$, Dolbeault-Perthame proved the global existence of
		a solution in a weak sense \cite{Dolbeault1}.
		When $\chi M=8\pi$, Blanchet-Carrillo-Masmoudi proved the global existence of free-energy solution
		under the assumption $(1+|x|^2)n_0\in L^1_{+}(\mathbb{R}^2)$ \cite{Blanchet1}.
		They also showed the global classical solutions can not exist and blow up in finite time for $\chi M>8\pi$
		under the assumption $(1+|x|^2)n_0\in L^1_{+}(\mathbb{R}^2)$.
		For the multi-species Patlak-Keller-Segel system (\ref{ini}),  define the total cell mass by
		$$M_1:=||n_{1}(t)||_{L^1}=||n_{1,\rm in}||_{L^1}\ {\rm and}\ M_2:=||n_{2}(t)||_{L^1}=||n_{2,\rm in}||_{L^1}.$$
		Espejo-Vilches-Conca gave a critical curve on the initial mass $M_1$ and $M_2$ planes, where one side of the solution globally exists and the other side explodes \cite{Espejo1}.
		He-Tadmor showed that the system (\ref{ini}) is globally well-posed under the sub-critical condition $\max\{\chi_1M_1,\chi_2M_2\}<8\pi$ \cite{he1}.
		
		An interesting question is to consider whether the stabilizing effect of the moving fluid can suppress the finite-time blow-up.
		Recently,  some progress has been made on this topic.
		For the  parabolic-elliptic PKS system, Kiselev-Xu suppressed the  blow-up by stationary relaxation
		enhancing flows and time-dependent Yao-Zlatos flows in $\mathbb{T}^d$ \cite{Kiselev1}.
		Bedrossian-He have studied the suppression of blow-up by shear flows in
		$\mathbb{T}^2$ for the 2-D parabolic-elliptic case,
		and for the 3-D parabolic-elliptic case, they investigated the suppression of blow-up by shear flows in $\mathbb{T}^3$ and $\mathbb{T}\times\mathbb{R}^2$ in \cite{Bedrossian1}.
		He investigated the suppression of blow-up by a large strictly monotone shear flow for the parabolic-parabolic PKS model in $\mathbb{T}\times\mathbb{R}$ \cite{he0}.
		Feng-Shi-Wang used the planar helical flows as transport
		flow to research the advective Kuramoto-Sivashinsky and
		Keller-Segel equations. They proved that when the amplitude of the flow is large enough, the $L^2$ norm of the solution is uniformly bounded in time \cite{Feng1}. For the coupled PKS-NS system, Zeng-Zhang-Zi considered the 2-D PKS-NS system near the Couette flow in $\mathbb{T}\times\mathbb{R}$, and they
		proved that if the Couette flow is sufficiently strong, the solution to the system stays globally regular \cite{zeng}.
		He considered the blow-up suppression for the parabolic-elliptic PKS-NS system in
		$\mathbb{T}\times\mathbb{R}$ with the coupling of buoyancy effects \cite{he05} for a class of small initial data.
		Li-Xiang-Xu studied the suppression of blow-up in PKS-NS system via the Poiseuille flow in $\mathbb{T}\times\mathbb{R}$,
		and they showd that if Poiseuille flow is sufficiently strong, the solution is global in a preprint \cite{Li0}.
		Besides, for the stability effect of buoyancy,
		Hu-Kiselev-Yao consider the blow-up suppression for the Patlak-Keller-Segel equation coupled with a fluid flow that obeys Darcy's law for incompressible porous media via buoyancy force \cite{Hu0}.
		Hu and Kiselev proved that when the coupling is large enough, the Keller-Segel equation coupled with Stokes-Boussinesq flow is globally well-posed \cite{Hu1}, see also the recent result by Hu \cite{Hu2023}.
		
		As mentioned by Zeng-Zhang-Zi ( \textbf{Remark 1.3, in \cite{zeng}} ):{\it ``It is very interesting to investigate the corresponding problems with bound-ary effects taken into account."} Exploring the potential suppression of blow-up when taking boundary effects into account seems to have not been taken into account in this issue yet.
		In this paper, we focus on the two-dimensional multi-species Patlak-Keller-Segel-Navier-Stokes system around the Poiseuille flow
		$(\ A(1-y^2), 0\ )$, where $(x,y)\in\mathbb{T}\times\mathbb{I}$.
		Consequently, the Navier-slip boundary condition 
		is imposed on the perturbation $u$.
		
		As is well known, the Poiseuille flow $U(y)=(\ A(1-y^2), 0\ ) $ is a stationary solution for the Navier-Stokes equations and the simplest non-trivial shear flow besides the Couette flow.
		In the meanwhile, $(\ 0,0,U(y)\ )$ is also a stationary solution of system (\ref{ini}).
		We introduce a perturbation $u$ around the Poiseuille flow on system (\ref{ini}),
		setting $u(t,x,y)=v(t,x,y)-U(y)$, then we get
		\begin{equation}\label{ini2}
			\left\{
			\begin{array}{lr}
				\partial_tn_1+A(1-y^2)\partial_x n_1-\triangle n_1=-\chi_1\nabla\cdot(n_1\nabla c)-u\cdot\nabla n_1, \\
				\partial_tn_2+A(1-y^2)\partial_x n_2-\triangle n_2=-\chi_2\nabla\cdot(n_2\nabla c)-u\cdot\nabla n_2, \\
				\triangle c+n_1+n_2-c=0, \\
				\partial_tu+A(1-y^2)\partial_x u+\left(
				\begin{array}{c}
					-2Ayu^2 \\
					0 \\
				\end{array}
				\right)
				-\triangle u+u\cdot\nabla u+\nabla P=(n_1+n_2)\nabla\phi, \\
				\nabla \cdot u=0,
			\end{array}
			\right.
		\end{equation}
		which is  assigned the Navier-slip boundary condition: \begin{equation}\label{slip_1}
			\left\{
			\begin{array}{lr}
				u^2(t,x,\pm1)=0, \\
				\partial_yu^1(t,x,\pm1)=0.
			\end{array}
			\right.
		\end{equation}
		To deal with the pressure $P$, we introduce the vorticity $\omega$ and the stream function $\Phi$ satisfying
		$$\omega=\partial_yu^1-\partial_xu^2\ {\rm and}\ u=(u^1, u^2)=\nabla^{\bot}\Phi,$$
		where $\nabla^{\bot}=\left(
		\begin{array}{c}
			\partial_y \\
			-\partial_x \\
		\end{array}
		\right)$.
		Then the vorticity $\omega$ satisfies
		$$\partial_t\omega+A(1-y^2)\partial_x\omega-2Au^2-\triangle\omega+u\cdot\nabla \omega=-(\partial_xn_1+\partial_xn_2).$$
		The Navier-slip boundary condition (\ref{slip_1}) yields
		$$\omega(t,x,\pm1)=0.$$
		
		After the time rescaling $t\mapsto\frac{t}{A}$, we rewrite the system (\ref{ini2}) into
		\begin{equation}\label{ini3}
			\left\{
			\begin{array}{lr}
				\partial_tn_1+(1-y^2)\partial_x n_1-\frac{1}{A}\triangle n_1
				=-\frac{1}{A}\Big( \chi_1\nabla\cdot(n_1\nabla c)+\nabla\cdot(un_1)\Big), \\
				\partial_tn_2+(1-y^2)\partial_x n_2-\frac{1}{A}\triangle n_2
				=-\frac{1}{A}\Big( \chi_2\nabla\cdot(n_2\nabla c)+\nabla\cdot (un_2)\Big), \\
				\triangle c+n_1+n_2-c=0, \\
				\partial_t\omega+(1-y^2)\partial_x\omega-\frac{1}{A}\triangle\omega+2\partial_x\triangle^{-1}\omega
				=-\frac{1}{A}(\partial_xn_1+\partial_xn_2+u\cdot\nabla \omega), \\
				u=\nabla^{\bot}\triangle^{-1}\omega.
			\end{array}
			\right.
		\end{equation}

		The main result of this paper is as follows.
		\begin{theorem}\label{result}
			Assume that the initial data $n_{1,{\rm in}}\in L^1\cap L^\infty(\mathbb{T}\times\mathbb{I})$,
			$n_{2,{\rm in}}\in L^1\cap L^\infty(\mathbb{T}\times\mathbb{I})$ and $u_{1,{\rm in}}\in H^1(\mathbb{T}\times\mathbb{I})$. The Dirichlet boundary conditions $n_1(t,x, \pm 1)=0$, $n_2(t,x, \pm 1)=0$ and $ c(t,x, \pm 1)=0$ are imposed.
			Then there exists a positive $A_2$ depending on $||n_{1,{\rm in}}||_{L^1\cap L^\infty(\mathbb{T}\times\mathbb{I})}$,
			$||n_{2,{\rm in}}||_{L^1\cap L^\infty(\mathbb{T}\times\mathbb{I})}$, $\chi_1$, $\chi_2$
			and $||u_{\rm in}||_{H^1(\mathbb{T}\times\mathbb{I})}$, such that if $A>A_2$, the solution of (\ref{ini3}) is global in time.
		\end{theorem}
		\begin{remark}
			For Dirichlet boundary conditions of  (\ref{ini3}) without the nonlinear term $u\cdot \nabla u$, local well-posedness of regular solutions is proved in \cite{Hu0}. Moreover, they proved the blow-up happens if the density becomes unbounded. Here we show that the density $n_1$ and $n_2$ will remain bounded under the suppression of shear flow. When considering the non-slip boundary condition of the velocity, it is still unknown and at this time it is much more involved, see \cite{chen}, for example.
		\end{remark}
		
		\begin{remark}
			Compared with the $\mathbb{T}\times\mathbb{R}$ case, we no longer need to introduce the weighted $L^2$ norm
			$||f||_{X}^2:=||f||_{L^2}^2+||yf||_{L^2}^2$
			presented by \cite{Coti1}.
		\end{remark}
		\begin{remark}
			In the classical multi-species PKS-NS system, the sensitivity  parameters $\chi_1$ and $\chi_2$ should be  positive real numbers. Following our proof line by line, we can find that there is no restriction on the symbol  of $\chi_1$ and $\chi_2$.
		\end{remark}
		\begin{remark}
			As we all know, the enhanced dissipation effect plays an important role in suppressing the blow-up of the solution and stabilizing the Navier-Stokes equations \cite{wei}.
			In $(x,y)\in\mathbb{T}\times\mathbb{I}$, the enhanced dissipation effect still plays an important role in suppressing the blow-up.
		\end{remark}
		
		The above Dirichlet boundary conditions  can be replaced with the usual Neumann boundary conditions.
		\begin{theorem}\label{result_00}
			Under the assumptions of \textbf{Theorem \ref{result}}, and the boundary conditions are replaced by $\partial_y n_1(t,x, \pm 1)=0$, $\partial_y n_2(t,x, \pm 1)=0$ and $\partial_y c(t,x, \pm 1)=0$.
			There exists a positive $A_3$ depending on $||n_{1,{\rm in}}||_{L^1\cap L^\infty(\mathbb{T}\times\mathbb{I})}$,
			$||n_{2,{\rm in}}||_{L^1\cap L^\infty(\mathbb{T}\times\mathbb{I})}$, $\chi_1$, $\chi_2$
			and $||u_{\rm in}||_{H^1(\mathbb{T}\times\mathbb{I})}$, such that if $A>A_3$, the solution of (\ref{ini3}) is global in time.
		\end{theorem}

		
		This paper is structured as follows:
		In Section \ref{sec_method}, a concise introduction to the method is presented.
		Section \ref{sec_time} is devoted to providing time-space estimates for bounding the non-zero modes of $n_1$, $n_2$, and $\omega$.
		Section \ref{sec_pri} a collection of elementary lemmas are offered, which are essential for the proof of \textbf{Proposition \ref{pro1}}.
		Finally, the proof is completed in Section \ref{sec_pro}.
		
		Here are some notations used in this paper.
		
		\noindent\textbf{Notations}:
		\begin{itemize}
			\item Throughout the paper, we denote by $C$ various ``harmless'' positive constants independent of $A$, $t$ and initial data. We would like to point out that $C$ may be different from line to line.
			\item For a given function $f=f(t,x,y)$,   we represent its zero mode and non-zero mode by
			$$P_0f=f_0=\frac{1}{|\mathbb{T}|}\int_{\mathbb{T}}f(t,x,y)dx,\ {\rm and}\ P_{\neq}f=f_{\neq}=f-f_0.$$
			Especially, we use $n_{k,0}$ and $n_{k,\neq}$ to represent the zero mode
			and non-zero mode of the cell density $n_k(k=1,2)$, respectively.
			\item We define the Fourier transform by
			\begin{equation}
				f(x,y)=\sum_{k\in\mathbb{Z}}\hat{f}(k,y){\rm e}^{ikx}, \nonumber
			\end{equation}
			where $\hat{f}(k,y)=\frac{1}{2\pi}\int_{\mathbb{T}}{f}(x,y){\rm e}^{-ikx}dx.$
			
			\item The norm of the $L^p$ space is defined by
			$$\|f\|_{L^p(\mathbb{T}\times\mathbb{I})}=(\int_{\mathbb{T}\times\mathbb{I}}|f|^p dxdy)^{\frac{1}{p}},$$
			and $\langle\cdot,\cdot\rangle$ denotes the standard scalar product.
			\item The time-space norm  $\|f\|_{L^qL^p}$ is defined by
			$$\|f\|_{L^qL^p}=\big\|  \|f\|_{L^p(\mathbb{T}\times\mathbb{I})}\ \big\|_{L^q(0,t)}.$$
			\item We define the norm $\|f\|_{X_{a}}$ by
			$$\|f\|_{X_{a}}^2
			=\|{\rm e}^{aA^{-\frac{1}{2}}t}f\|^2_{L^{\infty}L^{2}}
			+\frac{1}{A^{\frac{1}{2}}}\|{\rm e}^{aA^{-\frac{1}{2}}t}f\|^2_{L^{2}L^{2}}
			+\frac{1}{A}\|{\rm e}^{aA^{-\frac{1}{2}}t}\nabla f\|^2_{L^{2}L^{2}}.$$
			
		\end{itemize}

		\section{Methodology}\label{sec_method}
		
		The system (\ref{ini3}) has two components: the $x$-independent part and the $x$-dependent part.
		These two components have markedly different properties, as the $x$-independent part remains constant and does not change,
		while the $x$-dependent part varies with $x$.
		We decompose the system (\ref{ini3}) into $x$-independent part (\ref{ini3_1}) and $x$-dependent part (\ref{ini3_2}) as in \cite{zeng}:
		\begin{equation}\label{ini3_1}
			\left\{
			\begin{array}{lr}
				\partial_tn_{1,0}-\frac{1}{A}\triangle n_{1,0}
				=-\frac{\chi_1}{A}\Big( \nabla\cdot(n_{1,\neq}\nabla c_{\neq})_0
				+\partial_y(n_{1,0}\partial_yc_0)\Big)-\frac{1}{A}\nabla\cdot(u_{\neq}n_{1,\neq})_{0}, \\
				\partial_tn_{2,0}-\frac{1}{A}\triangle n_{2,0}
				=-\frac{\chi_2}{A}\Big(\nabla\cdot(n_{2,\neq}\nabla c_{\neq})_0+\partial_y(n_{2,0}\partial_yc_0)\Big)
				-\frac{1}{A}\nabla\cdot(u_{\neq}n_{2,\neq})_{0}, \\
				-\triangle c_0+c_0=n_{1,0}+n_{2,0}, \\
				\partial_t\omega_0-\frac{1}{A}\triangle\omega_0=-\frac{1}{A}\nabla\cdot (u_{\neq}\omega_{\neq})_0, \\
				u_0=(\partial_y(\partial_{yy})^{-1}\omega_0,0),
			\end{array}
			\right.
		\end{equation}
		and
		\begin{equation}\label{ini3_2}
			\left\{
			\begin{array}{lr}
				\partial_tn_{1,\neq}+(1-y^2)\partial_x n_{1,\neq}-\frac{1}{A}\triangle n_{1,\neq}=
				-\frac{\chi_1}{A}\Big( \nabla\cdot(n_{1,\neq}\nabla c_{\neq})_{\neq}
				+\nabla\cdot(n_{1,0}\nabla c_{\neq})
				+\partial_y(n_{1,\neq}\partial_yc_0)\Big)\\
				\qquad\qquad\qquad\qquad\qquad\quad-\frac{1}{A}\Big( \nabla\cdot(u_{\neq}n_{1,\neq})_{\neq}
				+\nabla\cdot(u_0n_{1,\neq})
				+\nabla\cdot(u_{\neq}n_{1,0})
				\Big),  \\
				\partial_tn_{2,\neq}+(1-y^2)\partial_x n_{2,\neq}-\frac{1}{A}\triangle n_{2,\neq}=
				-\frac{\chi_2}{A}\Big( \nabla\cdot(n_{2,\neq}\nabla c_{\neq})_{\neq}
				+\nabla\cdot(n_{2,0}\nabla c_{\neq})
				+\partial_y(n_{2,\neq}\partial_yc_0)\Big)\\
				\qquad\qquad\qquad\qquad\qquad\quad-\frac{1}{A}\Big( \nabla\cdot(u_{\neq}n_{2,\neq})_{\neq}
				+\nabla\cdot(u_0n_{2,\neq})
				+\nabla\cdot(u_{\neq}n_{2,0})
				\Big),  \\
				-\triangle c_{\neq}+c_{\neq}=n_{1,\neq}+n_{2,\neq}, \\
				\partial_t\omega_{\neq}+(1-y^2)\partial_x\omega_{\neq}-\frac{1}{A}\triangle\omega_{\neq}
				+2\partial_x\triangle^{-1}\omega_{\neq}
				=-\frac{1}{A}\Big( \nabla\cdot(u_{\neq}\omega_{\neq})_{\neq}
				+\nabla\cdot(u_0\omega_{\neq})
				+\nabla\cdot(u_{\neq}\omega_0)
				\Big)\\
				\qquad\qquad\qquad\qquad\qquad\qquad\quad-\frac{1}{A}(\partial_xn_{1,\neq}+\partial_xn_{2,\neq}), \\
				u_{\neq}=\nabla^{\bot}\triangle^{-1}\omega_{\neq}.
			\end{array}
			\right.
		\end{equation}
		
		We introduce a energy functional
		$$E(t)=\|n_{1,\neq}\|_{X_a}+\|n_{2,\neq}\|_{X_a}+\|\omega_{\neq}\|_{X_a}.$$
		Let's designate $T$ as the terminal point of the largest range $[0, T]$ such that the following
		hypothesis hold
		\begin{align}
			E(t)+||n_1||_{L^{\infty}L^{\infty}}+||n_2||_{L^{\infty}L^{\infty}}\leq 2E_{0}, \label{assumption_1}
		\end{align}
		for any $t\in[0, T]$, where $E_0$ will be calculated during the calculation.
		
		In the following sections, our objective is to prove the subsequent proposition, which will allow us to refine the previous hypotheses. By achieving this and ensuring the local well-posedness of the system (\ref{ini3}), we can confirm the global existence of the solution.
		\begin{proposition}{}\label{pro1}
			Under the assumption of \textbf{Theorem \ref{result}},
			there exists a positive constant $A_2$ depending on $M_1$, $M_2$, $\chi_1^2$, $\chi_2^2$, $||u_{\rm in}||_{H^1}$,
			$||n_{1, \rm in}||_{L^2\cap L^\infty}$ and $||n_{2,\rm in}||_{L^2\cap L^\infty}$, such that if $A>A_2$, there holds
			$$E(t)+||n_1||_{L^{\infty}L^{\infty}}+||n_2||_{L^{\infty}L^{\infty}}\leq E_{0},$$
			for all $t\in[0,T].$
		\end{proposition}

		First, we need to prove that  $E(t)$ satisfies $$E(t)\leq C\big(E_{\rm in}+\frac{C_0}{A^{\alpha}}E(t)\big),$$
		where $\alpha>0$ and $$E_{\rm in}=\|(n_{1,{\rm in}})_{\neq}\|_{X_a}
		+\|(n_{2,{\rm in}})_{\neq}\|_{X_a}+\|(\omega_{{\rm in}})_{\neq}\|_{X_a}.$$
		When $\frac{CC_0}{A^{\alpha}}<\frac{1}{2}$, we get $$E(t)\leq2CE_{\rm in}:=E_{0,1}.$$
		Then, using the Moser-Alikakos iteration \cite{zeng},
		we can prove that
		$$
		||n_1||_{L^{\infty}L^{\infty}}+||n_2||_{L^{\infty}L^{\infty}}\leq E_{0,2}$$
		At last, setting $E_{0}=E_{0,1}+E_{0,2},$ the \textbf{Proposition \ref{pro1}} holds.
		
		\section{Time-space estimates}\label{sec_time}
		\subsection{Main result of time-space estimates}
		In finite channel $\Omega=\mathbb{T}\times\mathbb{I}$, where $\mathbb{I}=(-1,1)$,
		the following time-space estimate can be used to bound the non-zero mode of $\omega$.
		\begin{lemma}[Proposition 4.2 in \cite{ding}]\label{time_result_1}
			Let $$\mathcal{L}_0=-\frac{1}{A}\triangle+(1-y^2)\partial_x+2\partial_x\triangle^{-1},$$
			and $f$ satisfies the following equation
			\begin{equation}\label{time_1}
				\begin{cases}
					(\partial_t+\mathcal{L}_0)f=\nabla\cdot f_1, \\
					f(t;x,\pm1)=0, \\
					f\big|_{t=0}=f_{\rm in}=f(0),
				\end{cases}
			\end{equation}
			then there exists a positive constant $a>0$, such that
			$$||f_{\neq}||_{X_{a}}^2\leq C\Big(||f_{\neq}(0)||_{L^2}^2
			+{A}||{\rm e}^{aA^{-\frac{1}{2}}t}(f_1)_{\neq}||_{L^2L^2}^2\Big).$$
		\end{lemma}
		
		In fact, \textbf{Lemma \ref{time_result_1}} only can be used to  bound the non-zero mode of $\omega$,
		and cannot be used to bound the non-zero mode of $n_1$ and $n_2$.
		In this section, we will prove the following time-space estimate according to the ideas of \cite{ding,Li1},
		which can be used to  bound the non-zero mode of $n_1$ and $n_2$.
		\begin{lemma}[]\label{time_result_2}
			Let $$\mathcal{L}=-\frac{1}{A}\triangle+(1-y^2)\partial_x,$$ and $f$ satisfies the following equation
			\begin{equation}\label{time_2}
				\begin{cases}
					(\partial_t+\mathcal{L})f=\nabla\cdot f_2, \\
					f(t;x,\pm1)=0, \\
					f\big|_{t=0}=f_{\rm in}=f_0,
				\end{cases}
			\end{equation}
			then we have
			$$||f_{\neq}||_{X_{a}}^2\leq C\Big(||(f_0)_{\neq}||_{L^2}^2
			+{A}||{\rm e}^{aA^{-\frac{1}{2}}t}(f_2)_{\neq}||_{L^2L^2}^2\Big).$$
		\end{lemma}
		
		The steps for proving \textbf{Lemma \ref{time_result_2}} are as follows:
		\begin{itemize}
			\item Taking the Fourier transformation in $x$, the linearized operator $\mathcal{L}$ becomes
			$$\widetilde{\mathcal{L}}=-\frac{1}{A}(\partial_y^2-|k|^2)+ik(1-y^2).$$
			\item Studying the resolvent of $$-\frac{1}{A}(\partial_y^2-|k|^2)f+ik(1-y^2-\lambda)f=F,$$
			and establishing the resolvent estimate
			$$A^{-\frac{1}{2}}|k|^{\frac{1}{2}}||f||_{L^2}\leq C||F||_{L^2}.$$
			\item Using the semigroup method to obtain the enhanced dissipative estimate $$\|{\rm e}^{-t\mathcal{L}}f_{\neq}\|_{L^2}\leq C{\rm e}^{-c'A^{-\frac{1}{2}}t-A^{-1}t}\|f_{\neq}\|_{L^2},$$
			and proving the time-space estimate.
			
		\end{itemize}
		\subsection{Resolvent estimate for the Orr-sommerfeld equation}
		In this subsection, we will establish the resolvent estimate for the Orr-sommerfeld equation
		\begin{equation}\label{resolvent_1}
			\begin{cases}
				-\frac{1}{A}(\partial_y^2-|k|^2)f+ik(1-y^2-\lambda)f=F, \\
				f(\pm1)=0.
			\end{cases}
		\end{equation}
		The main result of the resolvent estimate is as follows.
		\begin{proposition}\label{pro_time}
			Let $f\in H^2(\mathbb{I})$ be the solution of (\ref{resolvent_1}) with $\lambda\in \mathbb{R}$
			and $F\in L^2(\mathbb{I})$,
			then there holds $$A^{-\frac{1}{2}}|k|^{\frac{1}{2}}||f||_{L^2}\leq C||F||_{L^2}.$$
		\end{proposition}
		\begin{remark}
			The proof of \textbf{Proposition \ref{pro_time}} is divided into three parts,
			$\lambda>1$, $0\leq\lambda\leq1$, and $\lambda<0$.
			And we have divided the proof into three subsections.
		\end{remark}
		\subsubsection{Case of $\lambda>1$}
		By integration by parts, we get
		$$\big|{\rm Re}\langle F, f\rangle\big|
		=\big|\langle -\frac{1}{A}(\partial_y^2-|k|^2)f, f \rangle\big|
		=\frac{1}{A}(||f'||_{L^2}^2+|k|^2||f||_{L^2}^2),$$
		and $$\big|{\rm Im}\langle F, f\rangle\big|=|k|\int_{-1}^1(\lambda-1+y^2)|f|^2dy,$$
		which yields that
		\begin{equation}
			||f'||_{L^2}^2+||kf||_{L^2}^2\leq A||F||_{L^2}||f||_{L^2},\label{w_f}
		\end{equation}
		and
		$$\int_{-1}^1(\lambda-1+y^2)|f|^2dy\leq \frac{1}{|k|}||F||_{L^2}||f||_{L^2}.$$
		
		When $y\in[-1,1]\setminus(-\delta,\delta)$, we have
		\begin{equation}\label{temp_1}
			\lambda-1+y^2\geq\delta^2.
		\end{equation}
		Using (\ref{w_f}), (\ref{temp_1}), and $||f||^2_{L^\infty}\lesssim ||f||_{L^2}||f'||_{L^2}$,
		we conclude that
		\begin{equation}
			\begin{aligned}
				||f||^2_{L^2}&\leq||f||^2_{L^2([-1,1]\setminus(-\delta,\delta))}+2\delta||f||^2_{L^\infty} \\
				&\lesssim \int_{[-1,1]\setminus(-\delta,\delta)}\frac{1}{\lambda-1+y^2}(\lambda-1+y^2)|f|^2dy
				+\delta||f||_{L^2}||f'||_{L^2} \\
				&\lesssim \frac{1}{|k|\delta^2}||F||_{L^2}||f||_{L^2}
				+\delta A^{\frac{1}{2}}||F||^{\frac{1}{2}}_{L^2}||f||^{\frac{3}{2}}_{L^2}.\nonumber
			\end{aligned}
		\end{equation}
		
		Taking $\delta=(A^{-1}|k|^{-1})^{\frac{1}{4}}$, we get
		$$||f||_{L^2}\lesssim A^{\frac{1}{2}}|k|^{-\frac{1}{2}}||F||_{L^2}.$$
		\subsubsection{Case of $0\leq\lambda\leq1$}
		Let $-1\leq y_1\leq0\leq y_2\leq1$, such that $\lambda=1-y_1^2=1-y_2^2$. It is obvious that $y_1=-y_2$.
		The $L^2$ estimate of the interval $(-1,1)\setminus(y_1,y_2)$ is stated in the following lemma.
		\begin{lemma}\label{w_10}
			For any $\delta\in(0,1]$, there holds
			\begin{equation}
				\begin{aligned}
					||f||^2_{L^2((-1,1)\setminus(y_1,y_2))}&
					\leq C\Big(\frac{\|f\|_{L^2}||F||_{L^2}}{|k|(y_2-y_1+\delta)\delta}
					+\frac{||f'\bar{f}||_{L^{\infty}(B(y_1,\delta)\cup B(y_2,\delta))}}{A|k|(y_2-y_1+\delta)\delta} \\
					&+\frac{||f'||_{L^2}||f||_{L^\infty}}{A|k|\delta^{\frac{3}{2}}(y_2-y_1+\delta)}
					+\delta||f||^2_{L^\infty}\Big).\nonumber
				\end{aligned}
			\end{equation}
		\end{lemma}
		
		\begin{proof}
			By integration by parts, we get
			\begin{equation}
				\begin{aligned}
					&\Big|{\rm Im}\langle F, \frac{f}{1-y^2-\lambda}
					\chi_{(-1,1)\setminus(y_1-\delta,y_2+\delta)} \rangle\Big| \\
					=&\Big|{\rm Im}\langle -\frac{1}{A}(\partial_y^2-|k|^2)f+ik(1-y^2-\lambda)f,
					\frac{f}{1-y^2-\lambda}\chi_{(-1,1)\setminus(y_1-\delta,y_2+\delta)} \rangle\Big| \\
					=&\Big|\int_{(-1,1)\setminus(y_1-\delta,y_2+\delta)}k|f|^2dy
					+{\rm Im}\langle-\frac{1}{A}\partial_y^2f,
					\frac{f}{1-y^2-\lambda}\chi_{(-1,1)\setminus(y_1-\delta,y_2+\delta)}\rangle\Big| \\
					=&\Big|\int_{(-1,1)\setminus(y_1-\delta,y_2+\delta)}k|f|^2dy
					+{\rm Im}\Big(-\frac{f'\bar{f}}{A(1-y^2-\lambda)}\big|_{-1}^{y_1-\delta}
					-\frac{f'\bar{f}}{A(1-y^2-\lambda)}\big|^{1}_{y_2+\delta}\Big) \\
					&+{\rm Im}\Big(\frac{1}{A}
					\int_{(-1,1)\setminus(y_1-\delta,y_2+\delta)}
					\frac{2yf'\bar{f}}{(1-y^2-\lambda)^2}dy\Big)\Big|,\nonumber
				\end{aligned}
			\end{equation}
			which yields
			\begin{equation}\label{omega_1}
				\begin{aligned}
					|k|\int_{(-1,1)\setminus(y_1-\delta,y_2+\delta)}|f|^2dy
					&\leq \big\|\frac{f}{1-y^2-\lambda}\big\|_{L^2((-1,1)\setminus(y_1-\delta,y_2+\delta))}||F||_{L^2}
					+2\frac{||f'\bar{f}||_{L^{\infty}(B(y_1,\delta)\cup B(y_2,\delta))}}{A|y_1^2-(y_1-\delta)^2|} \\
					&+\frac{1}{A}||f'||_{L^2}||f||_{L^\infty}\big\|\frac{2y}{(1-y^2-\lambda)^2}
					\big\|_{L^2((-1,1)\setminus(y_1-\delta,y_2+\delta))} \\
					&\leq \frac{||f||_{L^2}||F||_{L^2}}{(y_2-y_1+\delta)\delta}
					+2\frac{||f'\bar{f}||_{L^{\infty}(B(y_1,\delta)\cup B(y_2,\delta))}}{A(y_2-y_1+\delta)\delta} \\
					&+\frac{1}{A}||f'||_{L^2}||f||_{L^\infty}\big\|\frac{2y}{(1-y^2-\lambda)^2}
					\big\|_{L^2((-1,1)\setminus(y_1-\delta,y_2+\delta))}.
				\end{aligned}
			\end{equation}
			
			Thanks to
			$$\big\|\frac{2y}{(1-y^2-\lambda)^2}
			\big\|_{L^2((-1,1)\setminus(y_1-\delta,y_2+\delta))}\lesssim\frac{1}{(y_2-y_1+\delta)\delta^{\frac{3}{2}}},$$
			one obtains
			\begin{equation}
				\begin{aligned}
					||f||^2_{L^2((-1,1)\setminus(y_1,y_2))}&
					\leq ||f||^2_{L^2((-1,1)\setminus(y_1-\delta,y_2+\delta))}+2\delta||f||^2_{L^\infty} \\
					&\lesssim\frac{\|f\|_{L^2}||F||_{L^2}}{|k|(y_2-y_1+\delta)\delta}
					+\frac{||f'\bar{f}||_{L^{\infty}(B(y_1,\delta)\cup B(y_2,\delta))}}{A|k|(y_2-y_1+\delta)\delta} \\
					&+\frac{||f'||_{L^2}||f||_{L^\infty}}{A|k|\delta^{\frac{3}{2}}(y_2-y_1+\delta)}
					+\delta||f||^2_{L^\infty}.\nonumber
				\end{aligned}
			\end{equation}
		\end{proof}
		
		The $L^2$ estimate of the interval $(y_1,y_2)$ is stated in the following lemma.
		\begin{lemma}\label{w_11}
			If $\delta\in(0,\frac{y_2-y_1}{4}]$, there holds
			\begin{equation}
				\begin{aligned}
					||f||^2_{L^2(y_1,y_2)}&
					\lesssim \frac{\|f\|_{L^2}||F||_{L^2}}{|k|(y_2-y_1)\delta}
					+\frac{||f'\bar{f}||_{L^{\infty}(B(y_1,\delta)\cup B(y_2,\delta))}}{A|k|(y_2-y_1)\delta} \\
					&+\frac{||f'||_{L^2}||f||_{L^\infty}}{A|k|\delta^{\frac{3}{2}}(y_2-y_1)}
					+\delta||f||^2_{L^\infty}.\nonumber
				\end{aligned}
			\end{equation}
			If $\delta\geq\frac{y_2-y_1}{4}$, there holds
			\begin{equation}
				\begin{aligned}
					||f||^2_{L^2(y_1,y_2)}&
					\lesssim  \delta||f||^2_{L^\infty}.\nonumber
				\end{aligned}
			\end{equation}
		\end{lemma}
		
		\begin{proof}
			
			By integration by parts, we have
			\begin{equation}
				\begin{aligned}
					&\Big|{\rm Im}\langle F, \frac{f}{1-y^2-\lambda}\chi_{(y_1+\delta,y_2-\delta)} \rangle\Big| \\
					=&\Big|{\rm Im}\langle -\frac{1}{A}(\partial_y^2-|k|^2)f+ik(1-y^2-\lambda)f,
					\frac{f}{1-y^2-\lambda}\chi_{(y_1+\delta,y_2-\delta)} \rangle\Big| \\
					=&\Big|\int_{(y_1+\delta,y_2-\delta)}k|f|^2dy
					+{\rm Im}\langle-\frac{1}{A}\partial_y^2f,
					\frac{f}{1-y^2-\lambda}\chi_{(y_1+\delta,y_2-\delta)}\rangle\Big| \\
					=&\Big|\int_{y_1+\delta}^{y_2-\delta}k|f|^2dy
					+{\rm Im}\Big(-\frac{f'\bar{f}}{A(1-y^2-\lambda)}\big|_{y_1+\delta}^{y_2-\delta}\Big)
					+{\rm Im}\Big(\frac{1}{A}\int_{y_1+\delta}^{y_2-\delta}
					\frac{2yf'\bar{f}}{(1-y^2-\lambda)^2}dy\Big)\Big|,\nonumber
				\end{aligned}
			\end{equation}
			which yields
			\begin{equation}
				\begin{aligned}
					|k|\int_{y_1+\delta}^{y_2-\delta}|f|^2dy
					&\leq\big\|\frac{f}{1-y^2-\lambda}\big\|_{L^2(y_1+\delta,y_2-\delta)}||F||_{L^2}
					+\frac{2||f'\bar{f}||_{L^{\infty}(B(y_1,\delta)\cup B(y_2,\delta))}}
					{A|y_1^2-(y_1-\delta)^2|} \\
					&+\frac{1}{A}||f'||_{L^2}||f||_{L^\infty}\big\|\frac{2y}{(1-y^2-\lambda)^2}
					\big\|_{L^2(y_1+\delta,y_2-\delta)}
					.\nonumber
				\end{aligned}
			\end{equation}
			
			It is obvious that
			\begin{equation}
				\begin{aligned}
					\big|\frac{1}{(1-y^2-\lambda)^2}\big|\lesssim \frac{1}{\delta^2(y_2-y_1)^2},\label{w_12}
				\end{aligned}
			\end{equation}
			for $y\in(y_1+\delta,y_2-\delta).$
			
			Using (\ref{w_12}) and $y_2-y_1=2y_2\geq4\delta$, we obtain
			\begin{equation}
				\begin{aligned}
					\big\|\frac{2y}{(1-y^2-\lambda)^2}\big\|_{L^2(y_1+\delta,y_2-\delta)}
					&\leq\frac{1}{\delta(y_2-y_1)}
					\big\|\frac{2y}{(1-y^2-\lambda)}\big\|_{L^2(y_1+\delta,y_2-\delta)} \\
					&\lesssim\frac{(y_2-\delta)^{\frac{1}{2}}}{\delta(y_2-y_1)}
					\Big(\int_{y_1+\delta}^{y_2-\delta}\frac{1}{(1-y^2-\lambda)^2}dy^2\Big)^\frac{1}{2} \\
					&\lesssim\frac{1}{\delta(y_2-y_1)}\Big(\frac{y_2-\delta}{\delta(2y_2-\delta)}\Big)^\frac{1}{2}
					\lesssim\frac{1}{\delta^{\frac{3}{2}} (y_2-y_1)}.
					\nonumber
				\end{aligned}
			\end{equation}
			
			Therefore, we conclude that
			\begin{equation}
				\begin{aligned}
					||f||^2_{L^2(y_1+\delta,y_2-\delta)}
					&\lesssim\frac{1}{|k|\delta(y_2-y_1)}||F||_{L^2}||f||_{L^2}
					+\frac{2||f'\bar{f}||_{L^{\infty}(B(y_1,\delta)\cup B(y_2,\delta))}}
					{A|k|\delta(y_2-y_1)} \\
					&+\frac{||f'||_{L^2}||f||_{L^\infty}}{A|k|\delta^{\frac{3}{2}}(y_2-y_1)},\nonumber
				\end{aligned}
			\end{equation}
			and
			\begin{equation}
				\begin{aligned}
					||f||^2_{L^2(y_1,y_2)}
					&\leq  ||f||^2_{L^2(y_1+\delta,y_2-\delta)}+2\delta||f||_{L^\infty}^2\\
					&\lesssim\frac{1}{|k|\delta(y_2-y_1)}||F||_{L^2}||f||_{L^2}
					+\frac{2||f'\bar{f}||_{L^{\infty}(B(y_1,\delta)\cup B(y_2,\delta))}}
					{A|k|\delta(y_2-y_1)} \\
					&+\frac{||f'||_{L^2}||f||_{L^\infty}}{A|k|\delta^{\frac{3}{2}}(y_2-y_1)}
					+\delta||f||_{L^\infty}^2.
					\nonumber
				\end{aligned}
			\end{equation}
			
		\end{proof}
		
		\begin{lemma}\label{lemma_dw}
			For any $A\geq1$, $\lambda\in[0,1]$ and $\delta\in(0,1]$, there holds
			$$||f'||^2_{L^\infty(B(y_1,\delta)\cup B(y_2,\delta))}
			\lesssim A^2|k|^2\delta^3(y_2-y_1-\delta)^2\delta||f||^2_{L^\infty}+A^2\delta||F||^2_{L^2}
			+\frac{1}{\delta^2}||f||^2_{L^\infty}.$$
		\end{lemma}
		\begin{proof}
			Using interpolation, we have
			$$||f'||_{L^\infty(B(y_1,\delta))}\lesssim\frac{1}{\delta}||f||_{L^\infty}
			+||f''||_{L^1(B(y_1,\delta))}.$$
			Using the H\"{o}lder's inequality and $$|k|^2\int_{B(y_1,\delta)}|f|dy
			\lesssim\delta^{\frac{1}{2}}|k|^2||f||_{L^2}\lesssim A\delta^{\frac{1}{2}}||F||_{L^2},$$
			we have
			\begin{equation}
				\begin{aligned}
					\int_{B(y_1,\delta)}|w''|dy
					&\lesssim A|k|\int_{B(y_1,\delta)}|(1-y^2-\lambda)f|dy
					+A\int_{B(y_1,\delta)}|F|dy+|k|^2\int_{B(y_1,\delta)}|f|dy \\
					&\lesssim A|k|\delta^2(y_2-y_1-\delta)||f||_{L^\infty}+A\delta^{\frac{1}{2}}||F||_{L^2}.
					\nonumber
				\end{aligned}
			\end{equation}
			The method for proving $||f'||_{L^\infty(B(y_2,\delta))}$ is similarly.
			
		\end{proof}
		
		Now, we start to prove
		$A^{-\frac{1}{2}}|k|^{\frac{1}{2}}||f||_{L^2}\leq C||F||_{L^2}$
		in the case of $\lambda\in[0,1]$.
		
		\noindent$\bullet$ \textbf{Case 1}: $A^{-\frac{1}{4}}|k|^{-\frac{1}{4}}\geq\frac{y_2-y_1}{4}.$
		
		Using \textbf{Lemma \ref{w_11}}, we have
		\begin{equation}
			\begin{aligned}
				||f||^2_{L^2(y_1,y_2)}
				\lesssim\delta||f||_{L^\infty}^2. \label{equ_16}
			\end{aligned}
		\end{equation}
		
		Combining \textbf{Lemma \ref{lemma_dw}} with $A^{-\frac{1}{4}}|k|^{-\frac{1}{4}}\geq\frac{y_2-y_1}{4},$
		one deduces
		\begin{equation}
			\begin{aligned}
				\delta^3||f'||^2_{L^\infty(B(y_1,\delta)\cup B(y_2,\delta))}
				&\lesssim A^2|k|^2\delta^6(y_2-y_1-\delta)^2\delta||f||^2_{L^\infty}+A^2\delta^4||F||^2_{L^2} \\
				&\lesssim \delta||f||^2_{L^\infty}+\frac{||F||^2_{L^2}}{A^{-1}|k|}, \nonumber
			\end{aligned}
		\end{equation}
		and
		\begin{equation}
			\begin{aligned}
				\frac{||f'\bar{f}||_{L^{\infty}(B(y_1,\delta)\cup B(y_2,\delta))}}{A|k|(y_2-y_1+\delta)\delta}
				&\lesssim \frac{\delta^2||f'||_{L^{\infty}(B(y_1,\delta)\cup B(y_2,\delta))}||f||_{L^{\infty}}}{A|k|\delta^4} \\
				&\lesssim \delta^3||f'||^2_{L^{\infty}(B(y_1,\delta)\cup B(y_2,\delta))}+\delta||f||^2_{L^\infty} \\
				&\lesssim \delta||f||^2_{L^\infty}+\frac{||F||^2_{L^2}}{A^{-1}|k|}. \nonumber
			\end{aligned}
		\end{equation}
		
		Thanks to (\ref{w_f}), we have
		\begin{equation}
			\begin{aligned}
				\frac{||f'||_{L^2}||f||_{L^\infty}}{A|k|\delta^{\frac{3}{2}}(y_2-y_1+\delta)}
				&\lesssim \delta||f'||_{L^2}\delta^{\frac{1}{2}}||f||_{L^\infty}
				\lesssim \delta^2||f'||^2_{L^2}+\delta||f||^2_{L^\infty} \\
				&\lesssim A\delta^2||F||_{L^2}||f||_{L^2}+\delta||f||^2_{L^\infty} \\
				&\lesssim \frac{||F||_{L^2}||f||_{L^2}}{A^{-\frac{1}{2}}|k|^{\frac{1}{2}}}
				+\delta||f||^2_{L^\infty}.\nonumber
			\end{aligned}
		\end{equation}
		Therefore,
		\begin{equation}\label{equ_17}
			\begin{aligned}
				||f||^2_{L^2((-1,1)\setminus(y_1,y_2))}&
				\lesssim \frac{||F||_{L^2}\|f\|_{L^2}}{|k|(y_2-y_1+\delta)\delta}
				+\frac{||f'\bar{f}||_{L^{\infty}(B(y_1,\delta)\cup B(y_2,\delta))}}{A|k|(y_2-y_1+\delta)\delta} \\
				&+\frac{||f'||_{L^2}||f||_{L^\infty}}{A|k|\delta^{\frac{3}{2}}(y_2-y_1+\delta)}
				+\delta||f||^2_{L^\infty} \\
				&\lesssim \frac{||F||_{L^2}||f||_{L^2}}{A^{-\frac{1}{2}}|k|^{\frac{1}{2}}}
				+\frac{||F||^2_{L^2}}{A^{-1}|k|}+\delta||f||^2_{L^\infty}.
			\end{aligned}
		\end{equation}
		Using the  Gagliardo-Nirenberg inequality and (\ref{w_f}), we get
		\begin{equation}\label{equ_18}
			\begin{aligned}
				\delta||f||^2_{L^\infty}
				\lesssim \delta||f||_{L^2}||f'||_{L^2}
				\lesssim \frac{|F||^{\frac{1}{2}}_{L^2}||f||^{\frac{3}{2}}_{L^2}}{A^{-\frac{1}{4}}|k|^{\frac{1}{4}}}.
			\end{aligned}
		\end{equation}
		
		Combining (\ref{equ_16}), (\ref{equ_17}) and (\ref{equ_18}),
		we conclude that
		\begin{equation}
			\begin{aligned}
				||f||^2_{L^2}
				\lesssim \frac{||F||_{L^2}||f||_{L^2}}{A^{-\frac{1}{2}}|k|^{\frac{1}{2}}}
				+\frac{||F||^2_{L^2}}{A^{-1}|k|}
				+\frac{|F||^{\frac{1}{2}}_{L^2}||f||^{\frac{3}{2}}_{L^2}}{A^{-\frac{1}{4}}|k|^{\frac{1}{4}}},\nonumber
			\end{aligned}
		\end{equation}
		which yields
		$$||f||_{L^2}\lesssim A^{\frac{1}{2}}|k|^{-\frac{1}{2}}||F||_{L^2}.$$
		
		\noindent$\bullet$ \textbf{Case 2}: $A^{-\frac{1}{4}}|k|^{-\frac{1}{4}}\leq\frac{y_2-y_1}{4}.$
		
		Setting $\delta^3(y_2-y_1)A|k|=1$, then we get
		\begin{equation}\label{delta_1}
			\begin{cases}
				0<\delta< \frac{y_2-y_1}{4}, \\
				0<\delta\lesssim A^{-\frac{1}{4}}|k|^{-\frac{1}{4}}.
			\end{cases}
		\end{equation}
		
		Combining \textbf{Lemma \ref{w_10}} with \textbf{Lemma \ref{w_11}}, we only need to estimate $||\omega||^2_{L^2(y_1,y_2)}$.
		
		Using $\delta^3(y_2-y_1)A|k|=1$ and (\ref{delta_1}),
		we have
		\begin{equation}
			\begin{aligned}
				\frac{\|f\|_{L^2}||F||_{L^2}}{|k|(y_2-y_1)\delta}
				\lesssim \frac{\|f\|_{L^2}||F||_{L^2}}{|k|(y_2-y_1)^{\frac{1}{2}}\delta^{\frac{3}{2}}}
				=\frac{\|f\|_{L^2}||F||_{L^2}}{A^{-\frac{1}{2}}|k|^{\frac{1}{2}}}, \label{cond2_1}
			\end{aligned}
		\end{equation}
		and
		\begin{equation}
			\begin{aligned}
				\delta||f||^2_{L^\infty}\lesssim\delta||f||_{L^2}||f'||_{L^2}
				\lesssim \delta A^{\frac{1}{2}}|F||^{\frac{1}{2}}_{L^2}||f||^{\frac{3}{2}}_{L^2}
				\lesssim \frac{|F||^{\frac{1}{2}}_{L^2}||f||^{\frac{3}{2}}_{L^2}}{A^{-\frac{1}{4}}|k|^{\frac{1}{4}}}.
				\label{cond2_2}
			\end{aligned}
		\end{equation}
		
		Using \textbf{Lemma \ref{lemma_dw}} and (\ref{delta_1}), we have
		\begin{equation}
			\begin{aligned}
				\delta^3||f'||^2_{L^\infty(B(y_1,\delta)\cup B(y_2,\delta))}
				&\lesssim A^2|k|^2\delta^6(y_2-y_1)^2\delta||f||^2_{L^\infty}+A^2\delta^4||F||^2_{L^2}
				+\delta||f||^2_{L^\infty} \\
				&\lesssim \frac{||F||^2_{L^2}}{A^{-1}|k|}
				+\frac{|F||^{\frac{1}{2}}_{L^2}||f||^{\frac{3}{2}}_{L^2}}{A^{-\frac{1}{4}}|k|^{\frac{1}{4}}},\nonumber
			\end{aligned}
		\end{equation}
		thus
		\begin{equation}
			\begin{aligned}
				\frac{||f'\bar{f}||_{L^{\infty}(B(y_1,\delta)\cup B(y_2,\delta))}}{A|k|(y_2-y_1)\delta}
				&\lesssim \delta^2||f'||_{L^{\infty}(B(y_1,\delta)\cup B(y_2,\delta))}||f||_{L^{\infty}} \\
				&\lesssim \delta^3||f'||^2_{L^{\infty}(B(y_1,\delta)\cup B(y_2,\delta))}+\delta||f||^2_{L^\infty} \\
				&\lesssim \frac{||F||^2_{L^2}}{A^{-1}|k|}+
				\frac{|F||^{\frac{1}{2}}_{L^2}||f||^{\frac{3}{2}}_{L^2}}{A^{-\frac{1}{4}}|k|^{\frac{1}{4}}}.
				\label{cond2_3}
			\end{aligned}
		\end{equation}
		
		Similar to \textbf{Case 1}, we have
		\begin{equation}
			\begin{aligned}
				\frac{||f'||_{L^2}||f||_{L^\infty}}{A|k|\delta^{\frac{3}{2}}(y_2-y_1)}
				&\lesssim \delta||f'||_{L^2}\delta^{\frac{1}{2}}||f||_{L^\infty}
				\lesssim \delta^2||f'||^2_{L^2}+\delta||f||^2_{L^\infty} \\
				&\lesssim \frac{||F||_{L^2}||f||_{L^2}}{A^{-\frac{1}{2}}|k|^{\frac{1}{2}}}
				+\frac{|F||^{\frac{1}{2}}_{L^2}||f||^{\frac{3}{2}}_{L^2}}{A^{-\frac{1}{4}}|k|^{\frac{1}{4}}}.
				\label{cond2_4}
			\end{aligned}
		\end{equation}
		
		Combining (\ref{cond2_1}), (\ref{cond2_2}), (\ref{cond2_3}) and (\ref{cond2_4}),
		we conclude that
		\begin{equation}
			\begin{aligned}
				||f||^2_{L^2}\leq2||f||^2_{L^2(y_1,y_2)}\lesssim \frac{||F||_{L^2}||f||_{L^2}}{A^{-\frac{1}{2}}|k|^{\frac{1}{2}}}
				+\frac{||F||^2_{L^2}}{A^{-1}|k|}
				+\frac{|F||^{\frac{1}{2}}_{L^2}||f||^{\frac{3}{2}}_{L^2}}{A^{-\frac{1}{4}}|k|^{\frac{1}{4}}},
			\end{aligned}
		\end{equation}
		which yields
		$$||f||_{L^2}\lesssim A^{\frac{1}{2}}|k|^{-\frac{1}{2}}||F||_{L^2}.$$
		
		The proof is complete.
		
		\subsubsection{Case of $\lambda\leq0$}
		By integration by parts, we have
		\begin{equation}
			\begin{aligned}
				&\Big|{\rm Im}\langle F, \frac{f}{1-y^2-\lambda}\chi_{(-1+\delta,1-\delta)} \rangle\Big| \\
				=&\Big|{\rm Im}\langle -\frac{1}{A}(\partial_y^2-|k|^2)f+ik(1-y^2-\lambda)f,
				\frac{f}{1-y^2-\lambda}\chi_{(-1+\delta,1-\delta)} \rangle\Big| \\
				=&\Big|\int_{(-1+\delta,1-\delta)}k|f|^2dy
				+{\rm Im}\langle-\frac{1}{A}\partial_y^2f,
				\frac{f}{1-y^2-\lambda}\chi_{(-1+\delta,1-\delta)}\rangle\Big| \\
				=&\Big|\int_{-1+\delta}^{1-\delta}k|f|^2dy
				+{\rm Im}\Big(-\frac{f'\bar{f}}{A(1-y^2-\lambda)}\big|_{-1+\delta}^{1-\delta}\Big)
				+{\rm Im}\Big(\frac{1}{A}\int_{-1+\delta}^{1-\delta}
				\frac{2yf'\bar{f}}{(1-y^2-\lambda)^2}dy\Big)\Big|,\nonumber
			\end{aligned}
		\end{equation}
		which yields
		\begin{equation}
			\begin{aligned}
				\int_{-1+\delta}^{1-\delta}|f|^2dy
				&\leq\frac{||F||_{L^2}}{|k|}\big\|\frac{f}{1-y^2-\lambda}\big\|_{L^2(-1+\delta,1-\delta)}
				+\frac{2||f'\bar{f}||_{L^{\infty}((-1,-1+\delta)\cup B(1-\delta,1))}}
				{A|k|(1-(1-\delta)^2-\lambda)} \\
				&+\frac{||f'||_{L^2}||f||_{L^\infty}}{A|k|}\big\|\frac{2y}{(1-y^2-\lambda)^2}
				\big\|_{L^2(-1+\delta,1-\delta)}
				.\label{lam_11}
			\end{aligned}
		\end{equation}
		
		Setting $\delta=A^{-\frac{1}{4}}|k|^{-\frac{1}{4}}$, we get
		\begin{equation}\label{xiaoyu1_1}
			\begin{aligned}
				\frac{||F||_{L^2}}{|k|}\big\|\frac{f}{1-y^2-\lambda}\big\|_{L^2(-1+\delta,1-\delta)}
				\leq\frac{||F||_{L^2}||f||_{L^2}}{A^{-\frac{1}{2}}|k|^{\frac{1}{2}}},
			\end{aligned}
		\end{equation}
		and
		\begin{equation}\label{xiaoyu1_2}
			\begin{aligned}
				\frac{||f'||_{L^2}||f||_{L^\infty}}{A|k|}\big\|\frac{2y}{(1-y^2-\lambda)^2}
				\big\|_{L^2(-1+\delta,1-\delta)}
				\lesssim\frac{\delta||f'||_{L^2}\delta^{\frac{1}{2}}||f||_{L^\infty}}{A|k|\delta^4}
				\lesssim \frac{||F||_{L^2}||f||_{L^2}}{A^{-\frac{1}{2}}|k|^{\frac{1}{2}}}+\delta||f||^2_{L^\infty}.
			\end{aligned}
		\end{equation}
		\begin{lemma}\label{lemma_xiaoyu1_1}
			For any $A\geq1$, $\lambda<0$ and $\delta=A^{-\frac{1}{4}}|k|^{-\frac{1}{4}}\in(0,1]$, there holds
			$$\frac{||f'\bar{f}||_{L^\infty((-1,-1+\delta)\cup(1-\delta,1))}}{A|k|(1-(1-\delta)^2-\lambda)}
			\lesssim  \frac{||F||^2_{L^2}}{A^{-1}|k|}
			+\delta||f||_{L^\infty}^2.$$
		\end{lemma}
		
		\begin{proof}
			We follow the proof of \textbf{Lemma \ref{lemma_dw}}.
			
			By interpolation, we have
			$$||f'||_{L^\infty((-1,-1+\delta)\cup(1-\delta,1))}\lesssim\frac{1}{\delta}||f||_{L^\infty}
			+||f''||_{L^1((-1,-1+\delta)\cup(1-\delta,1))}.$$
			Using the H\"{o}lder's inequality and $$|k|^2\int_{(-1,-1+\delta)\cup(1-\delta,1)}|f|dy
			\lesssim\delta^{\frac{1}{2}}|k|^2||f||_{L^2}\lesssim A\delta^{\frac{1}{2}}||F||_{L^2},$$
			we have
			\begin{equation}
				\begin{aligned}
					\int_{(-1,-1+\delta)\cup(1-\delta,1)}|w''|dy
					&\lesssim A|k|\int_{(-1,-1+\delta)\cup(1-\delta,1)}|(1-y^2-\lambda)f|dy \\
					&+A\int_{(-1,-1+\delta)\cup(1-\delta,1)}|F|dy+|k|^2\int_{(-1,-1+\delta)\cup(1-\delta,1)}|f|dy \\
					&\lesssim A|k|\delta(1-(1-\delta)^2-\lambda)||f||_{L^\infty}+A\delta^{\frac{1}{2}}||F||_{L^2}.
					\nonumber
				\end{aligned}
			\end{equation}
			Therefore, one obtains
			\begin{equation}
				\begin{aligned}
					\frac{||f'||_{L^\infty((-1,-1+\delta)\cup(1-\delta,1))}}{(1-(1-\delta)^2-\lambda)}
					\lesssim A|k|\delta||f||_{L^\infty}+\delta^{-3}||f||_{L^\infty}
					+A\delta^{-\frac{3}{2}}||F||_{L^2},\nonumber
				\end{aligned}
			\end{equation}
			and
			\begin{equation}
				\begin{aligned}
					\frac{||f'\bar{f}||_{L^\infty((-1,-1+\delta)\cup(1-\delta,1))}}{A|k|(1-(1-\delta)^2-\lambda)}
					&\lesssim \frac{||f'||_{L^\infty((-1,-1+\delta)\cup(1-\delta,1))}}{(1-(1-\delta)^2-\lambda)}
					\frac{||f||_{L^\infty}}{A|k|}\\
					&\lesssim  \frac{||F||_{L^2}\delta^{\frac{1}{2}}||f||_{L^\infty}}{\delta^{2}|k|}
					+\delta||f||_{L^\infty}^2 \\
					&\lesssim  \frac{||F||^2_{L^2}}{A^{-1}|k|}
					+\frac{|F||^{\frac{1}{2}}_{L^2}||f||^{\frac{3}{2}}_{L^2}}{A^{-\frac{1}{4}}|k|^{\frac{1}{4}}},\nonumber
				\end{aligned}
			\end{equation}
			with the help of (\ref{equ_18}).
			
		\end{proof}
		
		Combining (\ref{lam_11}), (\ref{xiaoyu1_1}), (\ref{xiaoyu1_2}), and \textbf{Lemma \ref{lemma_xiaoyu1_1}},
		we conclude that
		\begin{equation}
			\begin{aligned}
				||f||^2_{L^2}
				&\leq ||f||^2_{L^2(-1+\delta,1-\delta)}+2\delta||f||^2_{L^\infty}\\
				&\lesssim \frac{||F||_{L^2}||f||_{L^2}}{A^{-\frac{1}{2}}|k|^{\frac{1}{2}}}
				+\frac{||F||^2_{L^2}}{A^{-1}|k|}
				+\frac{|F||^{\frac{1}{2}}_{L^2}||f||^{\frac{3}{2}}_{L^2}}{A^{-\frac{1}{4}}|k|^{\frac{1}{4}}},
			\end{aligned}
		\end{equation}
		which yields
		$$||f||_{L^2}\lesssim A^{\frac{1}{2}}|k|^{-\frac{1}{2}}||F||_{L^2}.$$
		\subsection{Semigroup method}
		For any $f\in H^2{(\mathbb{I})}$, we have
		$${\rm Re}\langle\widetilde{\mathcal{L}}f,f\rangle=\frac{1}{A}(||f'||_{L^2}^2+|k|^2||f||_{L^2}^2)\geq0,$$
		thus $\widetilde{\mathcal{L}}$ is an accretive operator.
		It is obvious that the left open half-plane belongs to the resolvent set of $\widetilde{\mathcal{L}}$,
		so $\widetilde{\mathcal{L}}$ is a m-accretive operator.
		
		For the m-accretive operator, the following lemma holds.
		\begin{lemma}[\textbf{Theorem 1.3 in  \cite{wei1}}]\label{lemma_wei}
			Let $A$ be an m-accretive operator on a Hilbert space $X$. Then for any $t>0$, we have
			$$||{\rm e}^{-tA}||\leq{\rm e}^{-t\Psi(A)+\frac{\pi}{2}},$$
			where $$\Psi(A)=\inf\Big\{||(A-i\lambda)u||: u\in D(A), \lambda\in R,||u||=1\Big\}.$$
		\end{lemma}
		
		Based on the resolvent estimate, we find that
		$$\Psi(\mathcal{\widetilde{L}})\geq c'A^{-\frac{1}{2}}|k|^{\frac{1}{2}}+\frac{|k|^2}{A}.$$

		Using \textbf{Lemma \ref{lemma_wei}}, we can prove the following enhanced dissipation.
		\begin{proposition}[\textbf{Enhanced dissipation}]\label{lemma_ding1}
			For $A\geq1$, there exist constants $C, c'>0$ independent of $A$ such that
			$$\|{\rm e}^{-t\mathcal{L}}f_{\neq}\|_{L^2}\leq C{\rm e}^{-c'A^{-\frac{1}{2}}t-A^{-1}t}\|f_{\neq}\|_{L^2}.$$
		\end{proposition}
		\begin{proof}
			Using the Fourier series, we have
			\begin{equation}
				f=\sum_{k\in\mathbb{Z}}\hat{f}(k,y){\rm e}^{ikx}, \nonumber
			\end{equation}
			where $\hat{f}(k,y)=\frac{1}{2\pi}\int_{\mathbb{T}}{f}(x,y){\rm e}^{-ikx}dx.$
			
			Due to $${\rm e}^{-t\mathcal{L}}f_{\neq}=
			\sum_{k\neq0}{\rm e}^{-t\mathcal{\widetilde{L}}}\hat{f}(k,y){\rm e}^{ikx},$$
			thus we get  $$||{\rm e}^{-t\mathcal{L}}f_{\neq}||^2_{L^2}=
			2\pi\sum_{k\neq0}||{\rm e}^{-t\mathcal{\widetilde{L}}}\hat{f}(k,\cdot)||^2_{L^2}.$$
			
			Using \textbf{Lemma \ref{lemma_wei}}, we get
			$$||{\rm e}^{-t\mathcal{L}}f_{\neq}||_{L^2}\leq
			C{\rm e}^{-c'A^{-\frac{1}{2}}t-A^{-1}t}
			\sum_{k\neq0}||\hat{f}(k,\cdot)||_{L^2}=C{\rm e}^{-c'A^{-\frac{1}{2}}t-A^{-1}t}||f_{\neq}||_{L^2}.$$
			
		\end{proof}
		
		\begin{remark}
			Comparing \textbf{Proposition \ref{lemma_ding1}} with the \textbf{Proposition 4.1} of \cite{ding},
			we find the the enhanced dissipation of  $\mathcal{L}$
			is similar to the enhanced dissipation of $\mathcal{L}_0$.
			The only difference is the value of constant $c'$.
			By setting $c'=\min\{c,c'\}$ and  $c=\min\{c,c'\}$,
			then the enhanced dissipation of $\mathcal{L}$ and $\mathcal{L}_0$ becomes the same.
		\end{remark}
		
		We conclude the time-space estimate as following:
		\begin{proposition}[]
			Let $c'$ be in \textbf{Proposition \ref{lemma_ding1} and $a\in(0,c')$},
			then there holds
			$$\|{\rm e}^{aA^{-\frac{1}{2}}t}f\|^2_{L^{\infty}L^{2}}
			+\frac{1}{A^{\frac{1}{2}}}\|{\rm e}^{aA^{-\frac{1}{2}}t}f\|^2_{L^{2}L^{2}}
			+\frac{1}{A}\|{\rm e}^{aA^{-\frac{1}{2}}t}\nabla f\|^2_{L^{2}L^{2}}\lesssim||f_{\neq}(0)||_{L^2}^2
			+A\|{\rm e}^{aA^{-\frac{1}{2}}t} f_{2,\neq}\|^2_{L^{2}L^{2}}.$$
		\end{proposition}
		\begin{proof}
			The proof is similar to Proposition 6.3 of \cite{Li1}.
		\end{proof}

		\section{A Priori estimates}\label{sec_pri}
		\subsection{Elliptic estimates}
		The following  elliptic estimates are necessary.
		\begin{lemma}\label{ellip_0}
			Let $c_0$, $n_{1,0}$, and $n_{2,0}$ be the zero mode of $c$, $n_1$ and $n_2$, respectively, satisfying
			$$-\triangle c_0+c_0=n_{1,0}+n_{2,0},$$
			then there hold
			\begin{align}
				\|\partial_y^2c_0(t&)\|_{L^2}+\|\partial_yc_0(t)\|_{L^2}
				\leq 2\big(\|n_{1,0}(t)\|_{L^2}+\|n_{2,0}(t)\|_{L^2}\big), \\
				&\|\partial_yc_0(t)\|_{L^\infty}\leq 2\big(\|n_{1,0}(t)\|_{L^2}+\|n_{2,0}(t)\|_{L^2}\big),\label{c_inf}
			\end{align}
			and
			$$\|\partial_yc_0(t)\|_{L^4}\leq C\big(\|n_{1,0}(t)\|_{L^2}+\|n_{2,0}(t)\|_{L^2}\big),$$
			for any $t\geq0$.
		\end{lemma}
		\begin{proof}
			By a direct calculation, we have
			\begin{equation}
				\begin{aligned}
					\|\partial_y^2c_0(t)\|^2_{L^2}+\|\partial_yc_0(t)\|^2_{L^2}+\|c_0(t)\|^2_{L^2}
					&=\|n_{1,0}(t)\|^2_{L^2}+\|n_{2,0}(t)\|^2_{L^2}+2\langle n_{1,0},n_{2,0}\rangle \\
					&\leq 2(\|n_{1,0}(t)\|^2_{L^2}+\|n_{2,0}(t)\|^2_{L^2}),
					\nonumber
				\end{aligned}
			\end{equation}
			which indicates
			$$\|\partial_y^2c_0(t)\|_{L^2}+\|\partial_yc_0(t)\|_{L^2}\leq 2(\|n_{1,0}(t)\|_{L^2}+\|n_{2,0}(t)\|_{L^2}). $$
			Using the Gagliardo-Nirenberg inequality, we have
			$$\|\partial_yc_0(t)\|_{L^\infty}\leq
			2\|\partial_y^2c_0(t)\|^{\frac{1}{2}}_{L^2}\|\partial_yc_0(t)\|^{\frac{1}{2}}_{L^2}
			\leq 2(\|n_{1,0}(t)\|_{L^2}+\|n_{2,0}(t)\|_{L^2}),$$
			and
			$$\|\partial_yc_0(t)\|_{L^4}\leq
			C\|\partial_y^2c_0(t)\|^{\frac{1}{4}}_{L^2}\|\partial_yc_0(t)\|^{\frac{3}{4}}_{L^2}
			\leq C(\|n_{1,0}(t)\|_{L^2}+\|n_{2,0}(t)\|_{L^2}).$$
		\end{proof}
		
		\begin{lemma}\label{ellip}
			Let $c_{\neq}$, $n_{1,\neq}$, and $n_{2,\neq}$ be the non-zero mode of $c$, $n_1$ and $n_2$,
			respectively, satisfying
			$$-\triangle c_{\neq}+c_{\neq}=n_{1,\neq}+n_{2,\neq},$$
			then there hold
			\begin{align}
				\|\nabla^2c_{\neq}(t&)\|_{L^2}
				+\|\nabla c_{\neq}(t)\|_{L^2}\leq 2\big(\|n_{1,\neq}(t)\|_{L^2}+\|n_{2,\neq}(t)\|_{L^2}\big),
			\end{align}
			and
			\begin{equation}
				\|\nabla c_{\neq}(t)\|_{L^4}\leq C\big(\|n_{1,\neq}(t)\|_{L^2}+\|n_{2,\neq}(t)\|_{L^2}\big),
			\end{equation}
			for any $t\geq0$.
		\end{lemma}
		
		\begin{proof}
			A direct calculation shows
			$$\|\nabla^2c_{\neq}(t)\|^2_{L^2}+\|\nabla c_{\neq}(t)\|^2_{L^2}+\|c_{\neq}(t)\|^2_{L^2}
			= \|n_{1,\neq}(t)\|^2_{L^2}+\|n_{2,\neq}(t)\|^2_{L^2}+2\langle n_{1,\neq},n_{2,\neq}\rangle, $$
			which indicates
			$$\|\nabla^2c_{\neq}(t)\|_{L^2}+\|\nabla c_{\neq}(t)\|_{L^2}
			\leq 2\big(\|n_{1,\neq}(t)\|_{L^2}+\|n_{2,\neq}(t)\|_{L^2}\big). $$
			Using the Gagliardo-Nirenberg inequality, we have
			$$
			\|\nabla c_{\neq}(t)\|_{L^4}\leq
			C\|\nabla^2c_{\neq}(t)\|^{\frac{1}{2}}_{L^2}\|\nabla c_{\neq}(t)\|^{\frac{1}{2}}_{L^2}
			\leq C\big(\|n_{1,\neq}(t)\|_{L^2}+\|n_{2,\neq}(t)\|_{L^2}\big).
			$$
		\end{proof}
		\subsection{Transport term estimates}

		\begin{lemma}\label{lemma_0}
			Let $f$ be a function such that $f_{\neq}\in H^1(\mathbb{T}\times\mathbb{I})$
			and $\partial_xf_{\neq}\in H^1(\mathbb{T}\times\mathbb{I})$, there holds
			$$||f_{\neq}||_{L^2(\mathbb{T}\times\mathbb{I})}
			\leq||\partial_xf_{\neq}||_{L^2(\mathbb{T}\times\mathbb{I})}.$$
		\end{lemma}
		\begin{proof}
			By the Fourier series, we have
			\begin{equation}
				f=\sum_{k\in\mathbb{Z}}\hat{f}(k,y){\rm e}^{ikx}, \nonumber
			\end{equation}
			where $$\hat{f}(k,y)=\frac{1}{2\pi}\int_{\mathbb{T}}{f}(x,y){\rm e}^{-ikx}dx.$$
			Similarly,
			$$\partial_xf=\sum_{k\in\mathbb{Z}}\widehat{\partial_xf}(k,y){\rm e}^{ikx},$$
			where
			$\widehat{\partial_xf}(k,y)=\frac{1}{2\pi}\int_{\mathbb{T}}{\partial_xf}(x,y){\rm e}^{-ikx}dx.$
			
			When $k=0$, $\widehat{\partial_xf}(k,y)
			=\frac{1}{2\pi}\int_{\mathbb{T}}\partial_xf(x,y)dx=0.$
			Therefore, we have
			$$\partial_xf=\sum_{k\neq0}ik\hat{f}(k,y) {\rm e}^{ikx},$$
			and
			\begin{equation}
				f-f_0=\sum_{k\neq0}\hat{f}(k,\mu){\rm e}^{ikx}. \label{f_neq}
			\end{equation}
			
			A direct calculation shows
			\begin{equation}
				\begin{aligned}
					\|f_{\neq}\|_{L^2(\mathbb{T}\times \mathbb{I})}^2
					&=2\pi\sum_{k\neq0}\int_{\mathbb{I}}|\hat{f}(k,y)|^2dy.
					\nonumber
				\end{aligned}
			\end{equation}
			Accordingly,
			\begin{equation}
				\begin{aligned}
					\|\partial_xf_{\neq}\|_{L^2(\mathbb{T}\times \mathbb{I})}^2
					=2\pi\sum_{k\neq0}|k|^2\int_{\mathbb{I}}|\hat{f}(k,y)|^2dy\geq
					\|f_{\neq}\|_{L^2(\mathbb{T}\times \mathbb{I})}^2.
					\nonumber
				\end{aligned}
			\end{equation}
			
		\end{proof}
		
		\begin{lemma}[]\label{u_leq1}
			Let $u_{\neq}$ is the non-zero mode of $u$, satisfying $$u_{\neq}
			=\nabla^{\bot}\Phi_{\neq}, \triangle\Phi_{\neq}=\omega_{\neq}.$$
			Then there hold
			\begin{equation}
				\begin{aligned}
					&\|{\rm e}^{aA^{-\frac{1}{2}}t}u_{\neq}\|_{L^2L^2}\leq CA^{\frac{1}{4}}\|\omega_{\neq}\|_{X_a},\\
					&\|{\rm e}^{aA^{-\frac{1}{2}}t}u_{\neq}\|_{L^{\infty}L^2}\leq C\|\omega_{\neq}\|_{X_a}, \nonumber
				\end{aligned}
			\end{equation}
			and
			\begin{equation}
				\begin{aligned}
					\|{\rm e}^{aA^{-\frac{1}{2}}t}u_{\neq}\|_{L^2L^{\infty}}&\leq CA^{\frac{3}{10}}\|\omega_{\neq}\|_{X_a}. \nonumber
				\end{aligned}
			\end{equation}
		\end{lemma}
		
		\begin{proof}
			Using the Fourier series,
			we have
			\begin{equation}
				\begin{aligned}
					&u^1_{\neq}=\sum_{k\neq0}\partial_y\hat{\phi}_{\neq}(k,y){\rm e}^{ikx},\\
					u^2_{\neq}=-\sum_{k\neq0}&\widehat{\partial_x\phi}_{\neq}(k,y){\rm e}^{ikx}
					=-\sum_{k\neq0}ik\hat{\phi}_{\neq}(k,y){\rm e}^{ikx}, \nonumber
				\end{aligned}
			\end{equation}
			and
			$$\omega_{\neq}=\sum_{k\neq0}(\partial_y^2-k^2)\hat{\phi}_{\neq}(k,y){\rm e}^{ikx},$$
			where $\hat{\phi}_{\neq}(k,y)=\frac{1}{2\pi}\int_{\mathbb{T}}\phi_{\neq}(x,y){\rm e}^{-ikx}dx.$
			
			Thanks to the Navier-slip boundary conditions
			$$u^2(t,x,\pm1)=0, \partial_yu^1(t,x,\pm1)=0,$$
			for any $k\neq0$, we have
			$$\hat{\phi}(k,\pm 1)=\hat{\phi}''(k,\pm 1)=0,$$
			and
			$$\hat{\phi}_{\neq}(k,\pm 1)=\hat{\phi}_{\neq}''(k,\pm 1)=0.$$
			
			Therefore,
			\begin{equation}\label{w1}
				\begin{aligned}
					\|\omega_{\neq}\|_{L^2(\mathbb{T}\times\mathbb{I})}^2
					&=2\pi\int_{-1}^{1}\sum_{k\neq0}\Big((\partial_y^2-k^2)\hat{\phi}_{\neq}(k,y)\
					(\partial_y^2-k^2)\overline{\hat{\phi}}_{\neq}(k,y)\Big)dy \\
					&=\sum_{k\neq0}2\pi\Big(|k|^4\|\hat{\phi}_{\neq}(k,\cdot)\|^2_{L^2(\mathbb{I})}
					+2|k|^2\|\hat{\phi}'_{\neq}(k,\cdot)\|^2_{L^2(\mathbb{I})}
					+\|\hat{\phi}''_{\neq}(k,\cdot)\|^2_{L^2(\mathbb{I})}\Big),
				\end{aligned}
			\end{equation}
			\begin{equation}\label{u1}
				\begin{aligned}
					\|u^1_{\neq}\|_{L^2(\mathbb{T}\times\mathbb{I})}^2
					=2\pi\int_{-1}^{1}\sum_{k\neq0}\Big(\partial_y\hat{\phi}_{\neq}(k,y)\
					\partial_y\overline{\hat{\phi}}_{\neq}(k,y)\Big)dy
					=2\pi\sum_{k\neq0}\|\hat{\phi}'_{\neq}(k,\cdot)\|^2_{L^2(\mathbb{I})},
				\end{aligned}
			\end{equation}
			and
			\begin{equation}\label{u2}
				\begin{aligned}
					\|u^2_{\neq}\|_{L^2(\mathbb{T}\times\mathbb{I})}^2
					=2\pi\int_{-1}^{1}\sum_{k\neq0}\Big(|k|^2\hat{\phi}_{\neq}(k,y)\
					\overline{\hat{\phi}}_{\neq}(k,y)\Big)dy
					=2\pi\sum_{k\neq0}|k|^2\|\hat{\phi}_{\neq}(k,\cdot)\|^2_{L^2(\mathbb{I})}.
				\end{aligned}
			\end{equation}
			
			Combining (\ref{w1}), (\ref{u1}) and (\ref{u2}), we get
			$$||\nabla u_{\neq}||_{L^2}\leq C||\omega_{\neq}||_{L^2}.$$
			Using \textbf{Lemma \ref{lemma_0}}, we have
			\begin{equation}\label{u2_1}
				||u_{\neq}||_{L^2}\leq||\partial_x u_{\neq}||_{L^2}\leq ||\omega_{\neq}||_{L^2},
			\end{equation}
			which yields
			$$||{\rm e}^{aA^{-\frac{1}{2}}t}u_{\neq}||_{L^2L^2}\leq C||{\rm e}^{aA^{-\frac{1}{2}}t}\omega_{\neq}||_{L^2L^2}\leq CA^{\frac{1}{4}}\|\omega_{\neq}\|_{X_a},$$
			and
			$$||{\rm e}^{aA^{-\frac{1}{2}}t}u_{\neq}||_{L^\infty L^2}\leq C||{\rm e}^{aA^{-\frac{1}{2}}t}\omega_{\neq}||_{L^\infty L^2}\leq C\|\omega_{\neq}\|_{X_a}.$$
			
			Given that
			$$u_{\neq}=\sum_{k\neq0}\hat{u}_{\neq}(k,y){\rm e}^{ikx},$$
			using the Gagliardo-Nirenberg inequality
			$$||\hat{u}_{\neq}(k,y)||_{L^\infty_y}\lesssim
			||\hat{u}_{\neq}(k,\cdot)||_{L^2}^{\frac{1}{2}}||\partial_y\hat{u}_{\neq}(k,\cdot)||_{L^2}^{\frac{1}{2}},$$
			we have
			$$||u_{\neq}||_{L^\infty}\leq\sum_{k\neq0}||\hat{u}_{\neq}(k,y)||_{L^\infty_y}
			\lesssim\sum_{k\neq0}||\hat{u}_{\neq}(k,\cdot)||_{L^2}^{\frac{1}{2}}||\partial_y\hat{u}_{\neq}(k,\cdot)||_{L^2}^{\frac{1}{2}}.$$
			
			Thanks to the H\"{o}lder's inequality,  one obtains
			\begin{equation}
				\begin{aligned}
					||u_{\neq}||_{L^\infty}
					&\lesssim\Big(\sum_{k\neq0}|k|^{1+2\epsilon}||\hat{u}_{\neq}(k,\cdot)||_{L^2}||\partial_y\hat{u}_{\neq}(k,\cdot)||_{L^2}\Big)^{\frac{1}{2}}
					\Big(\sum_{k\neq0}\frac{1}{|k|^{1+2\epsilon}}\Big)^{\frac{1}{2}} \\
					&\lesssim\Big(\sum_{k\neq0}|k|^{2\epsilon}||k\hat{u}_{\neq}(k,\cdot)||^2_{L^2}\Big)^{\frac{1}{4}}\Big(\sum_{k\neq0}|k|^{2\epsilon}\|  \partial_y\hat{u}_{\neq}(k,\cdot)\|^2_{L^2}\Big)^{\frac{1}{4}} \\
					&\lesssim\Big(\sum_{k\neq0}|k|^{2\epsilon}||\widehat{\nabla u}_{\neq}(k,\cdot)||^2_{L^2}\Big)^{\frac{1}{2}},\nonumber
				\end{aligned}
			\end{equation}
			where $\epsilon\in(0,\frac{1}{2}].$
			
			Using the H\"{o}lder's inequality again, we have
			\begin{equation}
				\begin{aligned}
					||u_{\neq}||_{L^\infty}^2
					&\lesssim\sum_{k\neq0}\|k\widehat{\nabla u}_{\neq}(k,\cdot)\|^{2\epsilon}_{L^2}\|\widehat{\nabla u}_{\neq}(k,\cdot)\|^{2(1-\epsilon)}_{L^2} \\
					&\lesssim\Big(\sum_{k\neq0}\|k\widehat{\nabla u}_{\neq}(k,\cdot)\|^{2}_{L^2}\Big)^{\epsilon}
					\Big(\sum_{k\neq0}\|\widehat{\nabla u}_{\neq}(k,\cdot)\|^{2}_{L^2}\Big)^{1-\epsilon}\\
					&\lesssim\|\partial_x\nabla u_{\neq}\|_{L^2}^{2\epsilon}
					\|\nabla u_{\neq}\|_{L^2}^{2(1-\epsilon)}
					\lesssim\|\partial_x\omega_{\neq}\|_{L^2}^{2\epsilon}
					\|\omega_{\neq}\|_{L^2}^{2(1-\epsilon)}.
					\nonumber
				\end{aligned}
			\end{equation}
			
			Combining (\ref{w1}), (\ref{u1}) and (\ref{u2}), and setting $\epsilon=\frac{1}{5}$, we get
			$$		||u_{\neq}||_{L^\infty}
			\lesssim\|\partial_x\omega_{\neq}\|_{L^2}^{\frac{1}{5}}
			\|\omega_{\neq}\|_{L^2}^{\frac{4}{5}}.
			$$
			
			Therefore, we conclude that
			\begin{equation}
				\begin{aligned}
					||{\rm e}^{aA^{-\frac{1}{2}}t}u_{\neq}||_{L^2L^\infty}
					\lesssim
					\|{\rm e}^{aA^{-\frac{1}{2}}t}\partial_x\omega_{\neq}\|_{L^2L^2}^{\frac{1}{5}}
					\|{\rm e}^{aA^{-\frac{1}{2}}t}\omega_{\neq}\|_{L^2L^2}^{\frac{4}{5}}
					\lesssim A^{\frac{3}{10}}\|\omega\|_{X_a}.
					\nonumber
				\end{aligned}
			\end{equation}
			
			The proof is  complete.

		\end{proof}
		\subsection{Zero mode estimates}
		\begin{lemma}\label{priori}
			When the assumption (\ref{assumption_1}) holds, there exist a constant $C$ independent of
			$A$, $M_1$, $M_2$, $E_0$, $\chi_1$,  and  $\chi_2$,
			and a constant $A_0$ depending on $E_0$, such that if $A>A_0$, there hold
			\begin{equation}
				\begin{aligned}
					\|n_{1,0}\|^2_{L^\infty L^2}+\|n_{2,0}\|^2_{L^\infty L^2}
					\leq C\chi^2\Big(||n_{1,0}(0)||_{L^2}^2+||n_{2,0}(0)||_{L^2}^2
					+(M_1^4+M_2^4)+1\Big):=T_1,\label{t1}
				\end{aligned}
			\end{equation}
			\begin{equation}
				\begin{aligned}
					||\omega_0||_{L^\infty L^2}+\frac{1}{A^{\frac{1}{2}}}||\partial_y\omega_0||_{L^2L^2}
					\leq C\Big(||\omega_0(0)||_{L^2}+1\Big):=T_2,\label{t2}
				\end{aligned}
			\end{equation}
			and
			\begin{equation}
				\begin{aligned}
					||u_0^1||_{L^\infty L^\infty}\leq C(||(u_{\rm in})_0||_{L^2}
					+||(\omega_{\rm in})_0||_{L^2}+1):=T_3,\label{t3}
				\end{aligned}
			\end{equation}
			where $\chi^2={\rm max}\{\chi^2_1, \chi^2_2, 1\}.$
		\end{lemma}

		\begin{proof}
			Multiplying $n_{1,0}$ on $(\ref{ini3_1})_1$ and integrating in $y\in\mathbb{I}$, we have
			\begin{equation}
				\begin{aligned}
					\frac{1}{2}\frac{d}{dt}\|n_{1,0}\|^2_{L^2}
					+\frac{1}{A}\|\partial_yn_{1,0}\|^2_{L^2}
					&=\frac{\chi_1}{A}\int_{\mathbb{I}}(n_{1,\neq}\partial_y c_{\neq})_0\partial_yn_{1,0}dy
					+\frac{\chi_1}{A}\int_{\mathbb{I}}n_{1,0}\partial_yc_{0} \partial_yn_{1,0}dy \\
					&+\frac{1}{A}\int_{\mathbb{I}}(u_{\neq}^2 n_{1,\neq})_0\partial_yn_{1,0}dy.\nonumber
				\end{aligned}
			\end{equation}
			
			Thanks to H\"{o}lder's inequality and Young's inequality, we have
			\begin{equation}
				\begin{aligned}
					&\frac{\chi_1}{A}\int_{\mathbb{I}}(n_{1,\neq}\partial_y c_{\neq})_0\partial_yn_{1,0}dy \\
					&\leq \frac{1}{6A}||\partial_yn_{1,0}||_{L^2}^2
					+\frac{3\chi_1^2}{2A}||(n_{1,\neq}\partial_y c_{\neq})_0||_{L^2}^2
					\leq \frac{1}{6A}||\partial_yn_{1,0}||_{L^2}^2
					+\frac{3\chi_1^2}{4\pi A}||n_{1,\neq}||_{L^\infty}^2 ||\partial_y c_{\neq}||_{L^2}^2, \\
					&\frac{\chi_1}{A}\int_{\mathbb{I}}n_{1,0}\partial_yc_{0} \partial_yn_{1,0}dy \\
					&\leq \frac{1}{6A}||\partial_yn_{1,0}||_{L^2}^2
					+\frac{3\chi_1^2}{2A}||n_{1,0}\partial_yc_{0}||_{L^2}^2
					\leq \frac{1}{6A}||\partial_yn_{1,0}||_{L^2}^2
					+\frac{3\chi_1^2}{2A}||\partial_y c_{0}||_{L^\infty}^2||n_{1,0}||_{L^2}^2,
					\nonumber
				\end{aligned}
			\end{equation}
			and
			$$\frac{1}{A}\int_{\mathbb{I}}(u^2_{\neq}n_{1,\neq})_0\partial_yn_{1,0}dy
			\leq \frac{1}{6A}||\partial_yn_{1,0}||_{L^2}^2
			+\frac{3}{2A}||(u^2_{\neq}n_{1,\neq})_0||_{L^2}^2
			\leq \frac{1}{6A}||\partial_yn_{1,0}||_{L^2}^2
			+\frac{3}{4\pi A}||n_{1,\neq}||_{L^\infty}^2 ||u^2_{\neq}||_{L^2}^2.$$
			Thus, we conclude
			$$\frac{d}{dt}\|n_{1,0}\|^2_{L^2}+\frac{1}{A}\|\partial_yn_{1,0}\|^2_{L^2}
			\leq\frac{3\chi_1^2}{A}||\partial_y c_{0}||_{L^\infty}^2||n_{1,0}||_{L^2}^2
			+\frac{3}{2\pi A}||n_{1,\neq}||_{L^\infty}^2
			(||u^2_{\neq}||_{L^2}^2+\chi_1^2||\partial_y c_{\neq}||_{L^2}^2).$$
			
			Using (\ref{c_inf}) and Nash's inequality
			$$-||\partial_yn_{1,0}||_{L^2}^2\leq-\frac{||n_{1,0}||^6_{L^2}}{C||n_{1,0}||^4_{L^1}}
			\leq-\frac{||n_{1,0}||^6_{L^2}}{CM_1^4},$$
			we get
			\begin{equation}
				\begin{aligned}
					\frac{d}{dt}\|n_{1,0}\|^2_{L^2}
					\leq-\frac{||n_{1,0}||^6_{L^2}}{CAM_1^4}
					+\frac{3\chi_1^2}{A}||n_{1,0}||_{L^2}^2 \big(\|n_{1,0}(t)\|^2_{L^2}+\|n_{2,0}(t)\|^2_{L^2}\big) \\
					+\frac{3}{2\pi A}||n_{1,\neq}||_{L^\infty}^2 (||u^2_{\neq}||_{L^2}^2
					+\chi_1^2||\partial_y c_{\neq}||_{L^2}^2). \label{equ_dn1}
				\end{aligned}
			\end{equation}
			Similarly, we can get
			\begin{equation}
				\begin{aligned}
					\frac{d}{dt}\|n_{2,0}\|^2_{L^2}
					\leq-\frac{||n_{2,0}||^6_{L^2}}{CAM_2^4}
					+\frac{3\chi_2^2}{A}||n_{2,0}||_{L^2}^2 \big(\|n_{1,0}(t)\|^2_{L^2}+\|n_{2,0}(t)\|^2_{L^2}\big)
					\\+\frac{3}{2\pi A}||n_{2,\neq}||_{L^\infty}^2 (||u^2_{\neq}||_{L^2}
					+\chi_2^2||\partial_y c_{\neq}||_{L^2}^2). \label{equ_dn2}
				\end{aligned}
			\end{equation}
			Using (\ref{equ_dn1}) and (\ref{equ_dn2}), we find
			\begin{equation}
				\begin{aligned}
					\frac{d}{dt}(\|n_{1,0}\|^2_{L^2}+\|n_{2,0}\|^2_{L^2})
					\leq-\frac{(\|n_{1,0}\|^2_{L^2}+\|n_{2,0}\|^2_{L^2})^3}{CA(M_1^4+M_2^4)}+ \frac{3\chi^2}{A}(\|n_{1,0}\|^2_{L^2}+\|n_{2,0}\|^2_{L^2})^2  \\
					+\frac{3}{2\pi A}(||n_{1,\neq}||_{L^\infty}^2
					+||n_{2,\neq}||_{L^\infty}^2) (||u^2_{\neq}||_{L^2}^2+\chi^2||\partial_y c_{\neq}||_{L^2}^2) \\
					\leq-\frac{(\|n_{1,0}\|^2_{L^2}+\|n_{2,0}\|^2_{L^2})^2}{CA(M_1^4+M_2^4)}
					\Big(\|n_{1,0}\|^2_{L^2}+\|n_{2,0}\|^2_{L^2}-3C\chi^2(M_1^4+M_2^4)\Big) \\
					+\frac{3}{2\pi A}(||n_{1,\neq}||_{L^\infty}^2
					+||n_{2,\neq}||_{L^\infty}^2) (||u^2_{\neq}||_{L^2}^2
					+\chi^2||\partial_y c_{\neq}||_{L^2}^2),\label{equ_dn3}
				\end{aligned}
			\end{equation}
			where $$\chi^2={\rm max}\{\chi^2_1, \chi^2_2, 1\}.$$
			
			Define $G(t)$ by
			$$G(t)=\frac{3}{2\pi A}\int_0^t(||n_{1,\neq}||_{L^\infty}^2+||n_{2,\neq}||_{L^\infty}^2)
			(||u^2_{\neq}||_{L^2}^2+\chi^2||\partial_y c_{\neq}||_{L^2}^2)dt,$$
			then using \textbf{Lemma \ref{u_leq1}}, we get
			\begin{equation}
				\begin{aligned}
					G(t)
					&\leq\frac{3}{2\pi A}(||n_{1,\neq}||_{L^\infty L^\infty}^2+||n_{2,\neq}||_{L^\infty L^\infty}^2) (||u^2_{\neq}||_{L^2L^2}^2
					+\chi^2||\partial_y c_{\neq}||_{L^2L^2}^2) \\
					&\leq\frac{3}{2\pi A^{\frac{1}{2}}}(||n_{1,\neq}||_{L^\infty L^\infty}^2
					+||n_{2,\neq}||_{L^\infty L^\infty}^2) (||\omega_{\neq}||_{X_a}^2+\chi^2||\partial_y c_{\neq}||_{X_a}^2).
					\nonumber
				\end{aligned}
			\end{equation}
			Under the assumption (\ref{assumption_1}), there holds
			$$G(t)\leq \frac{C\chi^2E_0^4}{A^{\frac{1}{2}}}.$$
			When $A\geq E_0^8$, we have $$G(t)\leq C\chi^2.$$
			Rewrite (\ref{equ_dn3}) as
			\begin{equation}
				\begin{aligned}
					&\frac{d}{dt}(\|n_{1,0}\|^2_{L^2}+\|n_{2,0}\|^2_{L^2}-G(t))
					\\ &\leq-\frac{(\|n_{1,0}\|^2_{L^2}+\|n_{2,0}\|^2_{L^2})^2}{CA(M_1^4+M_2^4)}
					\Big(\|n_{1,0}\|^2_{L^2}+\|n_{2,0}\|^2_{L^2}-G(t)-3C\chi^2(M_1^4+M_2^4)\Big).\label{result_0_0}
				\end{aligned}
			\end{equation}
			It is obvious that
			\begin{equation}
				\begin{aligned}
					\|n_{1,0}\|^2_{L^2}+\|n_{2,0}\|^2_{L^2}-G(t)
					\leq 6C\chi^2\Big(||n_{1,0}(0)||_{L^2}^2+||n_{2,0}(0)||_{L^2}^2+(M_1^4+M_2^4)\Big),
					\ {\rm for}\ {\rm any}\ t\geq0.\nonumber 
				\end{aligned}
			\end{equation}
			Otherwise, there must exist $t=\check{t}>0$, such that
			\begin{equation}
				\begin{aligned}
					\|n_{1,0}(\check{t})\|^2_{L^2}+\|n_{2,0}(\check{t})\|^2_{L^2}-G(\check{t})= 6C\chi^2\Big(||n_{1,0}(0)||_{L^2}^2+||n_{2,0}(0)||_{L^2}^2+(M_1^4+M_2^4) \Big)\label{result_2},
				\end{aligned}
			\end{equation}
			and
			\begin{equation}
				\begin{aligned}
					\frac{d}{dt}\Big(\|n_{1,0}(\check{t})\|^2_{L^2}+\|n_{2,0}(\check{t})\|^2_{L^2}-G(\check{t})\Big)\geq 0.
					\label{result_3}
				\end{aligned}
			\end{equation}
			According to (\ref{result_0_0}) and (\ref{result_2}), we have
			\begin{equation}
				\begin{aligned}
					&\frac{d}{dt}\Big(\|n_{1,0}(\check{t})\|^2_{L^2}+\|n_{2,0}(\check{t})\|^2_{L^2}-G(\check{t})\Big) \\
					&\leq-\frac{6\chi^2(\|n_{1,0}(\check{t})\|^2_{L^2}+\|n_{2,0}(\check{t})\|^2_{L^2})^2}{CA(M_1^4+M_2^4)}
					\Big(||n_{1,0}(0)||_{L^2}^2+||n_{2,0}(0)||_{L^2}^2
					+\frac{1}{2}(M_1^4+M_2^4)\Big)<0.
					\label{result_4}
				\end{aligned}
			\end{equation}
			A contradiction arises between (\ref{result_3}) and (\ref{result_4}).
			
			When $A\geq E_0^8$, there holds
			\begin{equation}
				\begin{aligned}
					\|n_{1,0}(t)\|^2_{L^2}+\|n_{2,0}(t)\|^2_{L^2}&\leq C\chi^2\Big(||n_{1,0}(0)||_{L^2}^2
					+||n_{2,0}(0)||_{L^2}^2+(M_1^4+M_2^4)+1 \Big) \\
					&=C\chi^2\Big(||(n_{1,{\rm in}})_0||_{L^2}^2+||(n_{2,{\rm in}})_0||_{L^2}^2+(M_1^4+M_2^4)+1 \Big),
					\nonumber
				\end{aligned}
			\end{equation}
			for any $t\geq0$.
			
			Then we consider Eq. $(\ref{ini3_1})_4$, a direct calculation shows
			$$\frac{1}{2}\frac{d}{dt}||\omega_0||_{L^2}^2+\frac{1}{A}||\partial_y\omega_0||_{L^2}^2
			=\frac{1}{A}\int_{\mathbb{I}}(u^2_{\neq}\omega_{\neq})_0\partial_y\omega_0dy
			\leq\frac{1}{2A}||\partial_y\omega_0||^2_{L^2}+\frac{C}{A}||u^2_{\neq}\omega_{\neq}||_{L^2}^2.$$
			
			Using \textbf{Lemma \ref{u_leq1}}, we have
			\begin{equation}
				\begin{aligned}
					||\omega_0||_{L^\infty L^2}^2
					+\frac{1}{A}||\partial_y\omega_0||_{L^2L^2}^2
					&\leq 2||\omega_0(0)||_{L^2}^2+\frac{C}{A}
					||u^2_{\neq}||_{L^2L^\infty}^2||\omega_{\neq}||_{L^\infty L^2}^2\\
					&\leq 2||\omega_0(0)||_{L^2}^2+\frac{C}{A^{\frac{2}{5}}}||\omega_{\neq}||_{X_a}^4.\nonumber
				\end{aligned}
			\end{equation}
			If $A\geq E_0^{10}$,
			we arrive at
			\begin{equation}
				\begin{aligned}
					||\omega_0||_{L^\infty L^2}^2+\frac{1}{A}||\partial_y\omega_0||_{L^2L^2}^2
					\leq 2||\omega_0(0)||_{L^2}^2+C,\nonumber
				\end{aligned}
			\end{equation}
			which yields
			\begin{equation}
				\begin{aligned}
					||\omega_0||_{L^\infty L^2}+\frac{1}{A^{\frac{1}{2}}}||\partial_y\omega_0||_{L^2L^2}
					\leq C\Big(||(\omega_{\rm in})_0||_{L^2}+1\Big).\label{omega_0}
				\end{aligned}
			\end{equation}
			
			According to (\ref{ini2}), we learn that $u^1_0$ satisfies
			$$\partial_tu_0^1-\frac{1}{A}\partial_{yy}u_0^1=-\frac{1}{A}(u^2_{\neq}\partial_yu^1_{\neq})_0
			=-\frac{1}{A}\partial_y(u^2_{\neq}u^1_{\neq})_0.$$
			A direct calculation shows that
			\begin{equation}
				\begin{aligned}
					\frac{1}{2}\frac{d}{dt}||u_0^1||_{L^2}^2+\frac{1}{A}||\partial_{y}u_0^1||_{L^2}^2
					=\frac{1}{A}\int_{\mathbb{I}}(u^2_{\neq}u^1_{\neq})_0\partial_yu_0^1dy
					\leq\frac{1}{2A}||\partial_yu_0^1||_{L^2}^2+\frac{C}{A}||u^2_{\neq}u^1_{\neq}||_{L^2}^2,\nonumber
				\end{aligned}
			\end{equation}
			which yields
			\begin{equation}
				\begin{aligned}
					||u_0^1||_{L^\infty L^2}^2+\frac{1}{A}||\partial_{y}u_0^1||_{L^2L^2}^2
					\leq ||u_{0,{\rm in}}^1||_{L^2}^2+\frac{C}{A}||u^2_{\neq}u^1_{\neq}||_{L^2L^2}^2
					\leq ||u_{0,{\rm in}}^1||_{L^2}^2
					+\frac{C}{A}||u^2_{\neq}||_{L^\infty L^2}^2||u^1_{\neq}||_{L^2L^\infty}^2.\nonumber
				\end{aligned}
			\end{equation}
			If $A\geq E_0^{10}$, using \textbf{Lemma \ref{u_leq1}}, we get
			\begin{equation}
				\begin{aligned}
					||u_0^1||_{L^\infty L^2}^2+\frac{1}{A}||\partial_{y}u_0^1||_{L^2L^2}^2
					\leq ||u_{0,{\rm in}}^1||_{L^2}^2+\frac{C}{A^\frac{2}{5}}||\omega_{\neq}||_{X_a}^4
					\leq C(||u_{{\rm in}}||_{L^2}^2+1). \label{u_0}
				\end{aligned}
			\end{equation}
			Using (\ref{omega_0}), (\ref{u_0}) and the Gagliardo-Nirenberg inequality,
			we conclude that
			\begin{equation}
				\begin{aligned}
					||u_0^1||_{L^\infty L^\infty}
					&\leq C||u_0^1||_{L^\infty L^2}^{\frac{1}{2}}||\partial_yu_0^1||_{L^\infty L^2}^{\frac{1}{2}}
					=C||u_0^1||_{L^\infty L^2}^{\frac{1}{2}}||\omega_0||_{L^\infty L^2}^{\frac{1}{2}} \\
					&\leq C(||(u_{\rm in})_0||_{L^2}+||(\omega_{\rm in})_0||_{L^2}+1).
				\end{aligned}
			\end{equation}
			Letting $A_0=\max\{E_0^{8}, E_0^{10}\}$, we complete the proof.
			
		\end{proof}
		
		\section{Proof of Proposition \ref{pro1}}\label{sec_pro}
		\begin{lemma}\label{result_0}
			Under the conditions of \textbf{Lemma \ref{priori}},
			there exist a constant $E_0$ depending on $\|(\omega_{\rm in}){\neq}\|_{L^2}$,
			$\|(n_{1,\rm in})_{\neq}\|_{L^2}$ and $\|(n_{2,\rm in})_{\neq}\|_{L^2}$,
			and a positive constant $A_1$ depending on $E_0$, $E_\infty$, $M_1$, $M_2$,
			$\chi_1$, $\chi_2$, and $||u_{\rm in}||_{H^1}$,
			such that if $A>A_1$, there holds
			$$E(t)\leq E_{0}.$$
		\end{lemma}
		\noindent{\it Proof.}
		Applying \textbf{Lemma 3.1} to $(\ref{ini3_2})_4$,
		we get
		\begin{equation}
			\begin{aligned}
				\|\omega_{\neq}\|_{X_a}&\leq C\|\omega_{\neq}(0)\|_{L^2}
				+\frac{C}{A^{\frac{1}{2}}}\Big(\|{\rm e}^{aA^{-\frac{1}{2}}t} n_{1,\neq}\|_{L^2L^2}
				+\|{\rm e}^{aA^{-\frac{1}{2}}t} n_{2,\neq}\|_{L^2L^2}\Big) \\
				&+\frac{C}{A^{\frac{1}{2}}}\Big(\|{\rm e}^{aA^{-\frac{1}{2}}t}(u_{\neq}\omega_{\neq})_{\neq}\|_{L^2L^2}
				+\|{\rm e}^{aA^{-\frac{1}{2}}t}u_0^1\omega_{\neq}\|_{L^2L^2}
				+\|{\rm e}^{aA^{-\frac{1}{2}}t}u_{\neq}\omega_0\|_{L^2L^2}
				\Big).\label{pxn}
			\end{aligned}
		\end{equation}
		Due to
		\begin{equation}
			\begin{aligned}
				\|f_0\|^2_{L^2(\mathbb{T}\times\mathbb{I})}
				&=\int_{\mathbb{T}\times\mathbb{I}}(\frac{1}{\mathbb{|T|}}\int_{\mathbb{T}}f dx)^{2}dxdy
				\leq |\mathbb{T}|^{-1}\int_{\mathbb{I}}(\int_{\mathbb{T}}f dx)^{2}dy \\
				&\leq \int_{\mathbb{I}}\int_{\mathbb{T}}f^2 dxdy
				=\|f\|^2_{L^2(\mathbb{T}\times\mathbb{I})},\nonumber
			\end{aligned}
		\end{equation}
		thus
		\begin{equation}
			\begin{aligned}\|f_{\neq}\|_{L^2(\mathbb{T}\times\mathbb{I})}\leq\|f\|_{L^2(\mathbb{T}\times\mathbb{I})}
				+\|f_0\|_{L^2(\mathbb{T}\times\mathbb{I})}\leq C\|f\|_{L^2(\mathbb{T}\times\mathbb{I})}.\nonumber
			\end{aligned}
		\end{equation}
		
		According to \textbf{Lemma \ref{u_leq1}}, one deduces
		$$\bullet\quad  \|{\rm e}^{aA^{-\frac{1}{2}}t}u_0^1\omega_{\neq}\|_{L^2L^2}
		\leq ||u_0^1||_{L^{\infty}L^{\infty}}\|{\rm e}^{aA^{-\frac{1}{2}}t}\omega_{\neq}\|_{L^2L^2}
		\leq A^{\frac{1}{4}}\|\omega_{\neq}\|_{X_a}||u_0^1||_{L^{\infty}L^{\infty}}.$$
		$$\bullet\quad  \|{\rm e}^{aA^{-\frac{1}{2}}t} u_{\neq}\omega_0\|_{L^2L^2}
		\leq\|{\rm e}^{aA^{-\frac{1}{2}}t}u_{\neq}\|_{L^2 L^\infty}\|\omega_0\|_{L^\infty L^2}
		\leq A^{\frac{3}{10}}\|\omega_{\neq}\|_{X_a}\|\omega_0\|_{L^\infty L^2}.$$
		$$\bullet\quad   \|{\rm e}^{aA^{-\frac{1}{2}}t}(u_{\neq}\omega_{\neq})_{\neq}\|_{L^2L^2}
		\leq C\|{\rm e}^{aA^{-\frac{1}{2}}t}u_{\neq}\omega_{\neq}\|_{L^2L^2}
		\leq CA^{\frac{3}{10}}\|\omega_{\neq}\|_{X_a}\|\omega_{\neq}\|_{L^\infty L^2}.$$
		
		Using (\ref{t2}) and (\ref{t3}), we have
		\begin{equation}\label{omega_final}
			\begin{aligned}
				\|\omega_{\neq}\|_{X_a}&\leq C\|\omega_{\neq}(0)\|_{L^2}
				+\frac{C}{A^{\frac{1}{4}}}\Big(\| n_{1,\neq}\|_{X_a}
				+\| n_{2,\neq}\|_{X_a}\Big)
				+\frac{C}{A^{\frac{1}{4}}}||u_0^1||_{L^{\infty}L^{\infty}}\|\omega_{\neq}\|_{X_a}
				\\
				&+\frac{C}{A^\frac{1}{5}}\|\omega_{\neq}\|_{X_a}
				\big(\|\omega_{\neq}\|_{L^\infty L^2}+\|\omega_0\|_{L^\infty L^2}\big) \\
				&\leq C\|\omega_{\neq}(0)\|_{L^2}
				+\frac{C}{A^{\frac{1}{4}}}\Big(\| n_{1,\neq}\|_{X_a}
				+\| n_{2,\neq}\|_{X_a}\Big)
				+\frac{C(T_2+T_3+E_0)}{A^{\frac{1}{5}}}\|\omega_{\neq}\|_{X_a}.
			\end{aligned}
		\end{equation}
		
		Applying \textbf{Lemma 3.2} to $(\ref{ini3_2})_1$, we get
		\begin{equation}
			\begin{aligned}
				\|n_{1,\neq}\|_{X_a}&\leq C\|n_{1,\neq}(0)\|_{L^2}
				+\frac{C}{A^{\frac{1}{2}}} \|{\rm e}^{aA^{-\frac{1}{2}}t}(n_{1,\neq}\nabla c_{\neq})_{\neq}\|_{L^2L^2}
				+\frac{C}{A^{\frac{1}{2}}} \|{\rm e}^{aA^{-\frac{1}{2}}t}n_{1,0}\nabla c_{\neq}\|_{L^2L^2} \\
				&+\frac{C}{A^{\frac{1}{2}}}\|{\rm e}^{aA^{-\frac{1}{2}}t}n_{1,\neq}\partial_yc_{0}\|_{L^2L^2}
				+\frac{C}{A^{\frac{1}{2}}} \|{\rm e}^{aA^{-\frac{1}{2}}t}(u_{\neq}n_{1,\neq} )_{\neq}\|_{L^2L^2} \\
				&+\frac{C}{A^{\frac{1}{2}}}\|{\rm e}^{aA^{-\frac{1}{2}}t}u_{0}n_{1,\neq}\|_{L^2L^2}
				+\frac{C}{A^{\frac{1}{2}}} \|{\rm e}^{aA^{-\frac{1}{2}}t}u_{\neq}n_{1,0}\|_{L^2L^2}.\label{n1_neq}
			\end{aligned}
		\end{equation}
		We estimate nonlinear terms of (\ref{n1_neq}) as follows.
		
		$$\bullet\quad \|{\rm e}^{aA^{-\frac{1}{2}}t}(n_{1,\neq}\nabla c_{\neq})_{\neq}\|_{L^2L^2}
		\leq C\|{\rm e}^{aA^{-\frac{1}{2}}t}n_{1,\neq}\nabla c_{\neq}\|_{L^2L^2}
		\leq C||n_{1,\neq}||_{L^\infty L^\infty}\|{\rm e}^{aA^{-\frac{1}{2}}t}\nabla c_{\neq}\|_{L^2L^2}.$$
		$$\bullet\quad\|{\rm e}^{aA^{-\frac{1}{2}}t}n_{1,0}\nabla c_{\neq}\|_{L^2L^2}
		\leq C||n_{1,0}||_{L^\infty L^\infty}\|{\rm e}^{aA^{-\frac{1}{2}}t}\nabla c_{\neq}\|_{L^2L^2}. $$
		$$\bullet\quad \|{\rm e}^{aA^{-\frac{1}{2}}t}n_{1,\neq}\partial_yc_{0}\|_{L^2L^2}
		\leq\|\partial_yc_{0}\|_{L^\infty L^\infty}\|{\rm e}^{aA^{-\frac{1}{2}}t}n_{1,\neq}\|_{L^2L^2}.$$
		$$\bullet\quad\|{\rm e}^{aA^{-\frac{1}{2}}t}(u_{\neq}n_{1,\neq} )_{\neq}\|_{L^2L^2}
		\leq C\|{\rm e}^{aA^{-\frac{1}{2}}t}u_{\neq}n_{1,\neq}\|_{L^2L^2}
		\leq C\|n_{1,\neq}\|_{L^\infty L^\infty}\|{\rm e}^{aA^{-\frac{1}{2}}t}u_{\neq}\|_{L^2L^2}.$$
		$$\bullet\quad \|{\rm e}^{aA^{-\frac{1}{2}}t}u_{0}n_{1,\neq}\|_{L^2L^2}
		\leq\|u_{0}\|_{L^\infty L^\infty}\|{\rm e}^{aA^{-\frac{1}{2}}t}n_{1,\neq}\|_{L^2L^2}.$$
		$$\bullet\quad \|{\rm e}^{aA^{-\frac{1}{2}}t}u_{\neq}n_{1,0}\|_{L^2L^2}
		\leq \|n_{1,0}\|_{L^\infty L^\infty}\|{\rm e}^{aA^{-\frac{1}{2}}t}u_{\neq}\|_{L^2L^2}.$$
		Using (\ref{t1}), (\ref{t3}), \textbf{Lemma \ref{ellip_0}} and \textbf{Lemma \ref{u_leq1}}, we have
		\begin{equation}
			\begin{aligned}
				\|n_{1,\neq}\|_{X_a}\leq C\Big(\|n_{1,\neq}(0)\|_{L^2}
				+\frac{E_\infty}{A^{\frac{1}{4}}}||\omega_{\neq}||_{X_a}
				+\frac{T_1+T_3} {A^{\frac{1}{4}}}||n_{1,\neq}||_{X_a}
				+\frac{E_\infty}{A^{\frac{1}{4}}}(||n_{1,\neq}||_{X_a}+||n_{2,\neq}||_{X_a})\Big).\label{n1_final}
			\end{aligned}
		\end{equation}
		
		Applying \textbf{Lemma 3.2} to $(\ref{ini3_2})_2$, we have
		\begin{equation}
			\begin{aligned}
				\|n_{2,\neq}\|_{X_a}&\leq C\|n_{2,\neq}(0)\|_{L^2}
				+\frac{C}{A^{\frac{1}{2}}} \|{\rm e}^{aA^{-\frac{1}{2}}t}(n_{2,\neq}\nabla c_{\neq})_{\neq}\|_{L^2L^2}
				+\frac{C}{A^{\frac{1}{2}}} \|{\rm e}^{aA^{-\frac{1}{2}}t}n_{2,0}\nabla c_{\neq}\|_{L^2L^2} \\
				&+\frac{C}{A^{\frac{1}{2}}}\|{\rm e}^{aA^{-\frac{1}{2}}t}n_{2,\neq}\partial_yc_{0}\|_{L^2L^2}
				+\frac{C}{A^{\frac{1}{2}}} \|{\rm e}^{aA^{-\frac{1}{2}}t}(u_{\neq}n_{2,\neq} )_{\neq}\|_{L^2L^2} \\
				&+\frac{C}{A^{\frac{1}{2}}}\|{\rm e}^{aA^{-\frac{1}{2}}t}u_{0}n_{2,\neq}\|_{L^2L^2}
				+\frac{C}{A^{\frac{1}{2}}} \|{\rm e}^{aA^{-\frac{1}{2}}t}u_{\neq}n_{2,0}\|_{L^2L^2}.\nonumber
			\end{aligned}
		\end{equation}
		Similarly, we get
		\begin{equation}
			\begin{aligned}
				\|n_{2,\neq}\|_{X_a}\leq C\Big(\|n_{2,\neq}(0)\|_{L^2}
				+\frac{E_\infty}{A^{\frac{1}{4}}}||\omega_{\neq}||_{X_a}
				+\frac{T_1+T_3} {A^{\frac{1}{4}}}||n_{2,\neq}||_{X_a}
				+\frac{E_\infty}{A^{\frac{1}{4}}}(||n_{1,\neq}||_{X_a}+||n_{2,\neq}||_{X_a})\Big).\label{n2_final}
			\end{aligned}
		\end{equation}
		Using (\ref{omega_final}), (\ref{n1_final}) and  (\ref{n2_final}), one deduces
		$$E(t)\leq C\Big(\|(\omega_{\rm in})_{\neq}\|_{L^2}
		+\|(n_{1,\rm in})_{\neq}\|_{L^2}+\|(n_{2,\rm in})_{\neq}\|_{L^2}
		+\frac{1+T_1+T_2+T_3+E_0+E_\infty}{A^\frac{1}{5}}E(t)\Big).$$
		
		If $A\geq [2C(1+T_1+T_2+T_3+E_0+E_\infty)]^{5}:=A_1$,
		we conclude that $$E(t)\leq2C\Big(\|(\omega_{\rm in})_{\neq}\|_{L^2}
		+\|(n_{1,\rm in})_{\neq}\|_{L^2}+\|(n_{2,\rm in})_{\neq}\|_{L^2}\Big)=:E_{0,1}.$$
		
		\begin{lemma}\label{result_1}
			There exists a positive constant $E_\infty$ depending on $E_{0}$, $\chi_1$, $\chi_2$, $M_1$, $M_2$,
			$\|\nabla c_{\rm in}\|_{L^4}$, $\|n_{1,\rm in}\|_{L^2\cap L^{\infty}}$
			and $\|n_{2,\rm in}\|_{L^2\cap L^{\infty}}$, such that
			$$\|n_1\|_{L^{\infty}L^{\infty}}+\|n_2\|_{L^{\infty}L^{\infty}}\leq E_{0,2}.$$
		\end{lemma}
		\begin{proof}
			For
			\begin{equation}
				\begin{aligned}
					\partial_tn_1+(1-y^2)\partial_x n_1-\frac{1}{A}\triangle n_1
					=-\frac{\chi_1}{A}\nabla\cdot(n_1\nabla c)
					-\frac{1}{A}u\cdot\nabla n_1,\nonumber
				\end{aligned}
			\end{equation}
			multiply $2pn_1^{2p-1}$ where $p=2^j$ and $j\geq1$,
			and integrate the resulting equation over $\mathbb{T}\times\mathbb{I}$.
			By integration by parts, we get
			\begin{equation}
				\begin{aligned}
					\frac{d}{dt}||n_1^p||^2_{L^2}+\frac{2(2p-1)}{Ap}||\nabla n_1^{p}||_{L^2}^2
					&=\frac{2(2p-1)\chi_1}{A}\int_{\mathbb{T}\times\mathbb{I}}n_1^{p}\nabla c\cdot\nabla n_1^{p}dxdy \\
					&\leq\frac{2(2p-1)\chi_1}{A}||n_1^p\nabla c||_{L^2}||\nabla n_1^p||_{L^2} \\
					&\leq\frac{(2p-1)}{Ap}||\nabla n_1^p||_{L^2}^2
					+\frac{(2p-1)p\chi_1^2}{A}||n_1^p\nabla c||_{L^2}.
				\end{aligned}
			\end{equation}
			Using the Nash inequality
			\begin{equation}
				\begin{aligned}
					||n_1^p\nabla c||_{L^2}^2
					\leq C||n_1^p||_{L^4}^2||\nabla c||_{L^4}^2
					\leq C(||n_1^p||_{L^2}^2||\nabla n_1^p||_{L^2}^2+||n_1^p||_{L^2}^2)||\nabla c||_{L^4}^2, \nonumber
				\end{aligned}
			\end{equation}
			one deduces
			\begin{equation}
				\begin{aligned}
					&\frac{d}{dt}||n_1^p||^2_{L^2}+\frac{2(2p-1)}{Ap}||\nabla n_1^{p}||_{L^2}^2 \\
					&\leq\frac{(2p-1)}{Ap}||\nabla n_1^p||_{L^2}^2
					+\frac{C(2p-1)p\chi_1^2}{A}(||n_1^p||_{L^2}^2||\nabla n_1^p||_{L^2}^2+||n_1^p||_{L^2}^2)||\nabla c||_{L^4}^2 \\
					&\leq\frac{5(2p-1)}{4Ap}||\nabla n_1^p||_{L^2}^2
					+\frac{C(2p-1)p^3\chi_1^4}{A}||n_1^p||_{L^2}^2||\nabla c||_{L^4}^4
					+\frac{C(2p-1)p\chi_1^2}{A}||n_1^p||_{L^2}^2||\nabla c||_{L^4}^2.
				\end{aligned}
			\end{equation}
			
			Consequently,
			\begin{equation}
				\begin{aligned}
					\frac{d}{dt}||n_1^p||^2_{L^2}+\frac{2p-1}{2pA}||\nabla n_1^{p}||_{L^2}^2\leq
					\frac{Cp^4\chi_1^4}{A}||n_1^p||_{L^2}^2||\nabla c||_{L^4}^4
					+\frac{Cp^4\chi_1^4}{A}||n_1^p||_{L^2}^2||\nabla c||_{L^4}^2,
					\nonumber
				\end{aligned}
			\end{equation}
			which implies that
			\begin{equation}
				\begin{aligned}
					\frac{d}{dt}||n_1^p||^2_{L^2}+\frac{1}{2A}||\nabla n_1^{p}||_{L^2}^2\leq
					\frac{Cp^4\chi^4}{A}||n_1^p||_{L^2}^2(||\nabla c||_{L^4}^4+1).
					\label{58}
				\end{aligned}
			\end{equation}
			
			Using the Nash's inequality \begin{equation}
				\begin{aligned}
					||n_1^p||_{L^2}
					\leq C(||n_1^p||_{L^1}^{\frac{1}{2}}||\nabla n_1^p||_{L^2}^{\frac{1}{2}}+||n_1^p||_{L^1}),
					\nonumber
				\end{aligned}
			\end{equation}
			we infer from (\ref{58}) that
			\begin{equation}
				\begin{aligned}
					\frac{d}{dt}||n_1^p||^2_{L^2}
					\leq
					-\frac{||n_1^p||_{L^2}^4}{2AC||n_1^p||_{L^1}^2}
					+\frac{||n_1^p||_{L^1}^2}{2A}
					+\frac{Cp^4\chi^4}{A}||n_1^p||_{L^2}^2(||\nabla c||_{L^{\infty}L^4}^4+1).
					\nonumber
				\end{aligned}
			\end{equation}
			Similarly, we obtain
			\begin{equation}
				\begin{aligned}
					\frac{d}{dt}||n_2^p||^2_{L^2}
					\leq
					-\frac{||n_2^p||_{L^2}^4}{2AC||n_2^p||_{L^1}^2}
					+\frac{||n_2^p||_{L^1}^2}{2A}
					+\frac{Cp^4\chi^4}{A}||n_2^p||_{L^2}^2(||\nabla c||_{L^{\infty}L^4}^4+1).
					\nonumber
				\end{aligned}
			\end{equation}

			Applying \textbf{ Lemma \ref{ellip}}  and \textbf{Lemma \ref{priori}}, there holds
			$$ ||\nabla c||_{L^{\infty}L^4}\leq||\partial_y c_0||_{L^{\infty}L^4}
			+||\nabla c_{\neq}||_{L^{\infty}L^4}\leq C(T_1+E_{0}).$$
			
			Therefore,
			\begin{equation}
				\begin{aligned}
					\frac{d}{dt}||n_k^p||^2_{L^2}
					\leq-\frac{||n_k^p||^4_{L^2}}{2CA||n_k^p||_{L^1}^2}
					+\frac{||n_k^p||_{L^1}^2}{2A}
					+\frac{Cp^4\chi^4}{A}||n_k^p||^2_{L^2}(T_1^4+E_{0}^4+1),
				\end{aligned}
			\end{equation}
			where $k=1,2$.
			
			By applying \textbf{Lemma A.1} of \cite{zeng}, one deduces
			\begin{equation}
				\begin{aligned}
					\sup_{t\geq 0}||n_1^p||_{L^2}^2\leq \max\Big\{6C^2(T_1^4+E_{0}^4+1)\chi^4p^4\sup_{t\geq0}||n_1^p||^2_{L_1}, 2||n_{1,\rm in}^p||_{L^2}^2,1 \Big\},\label{n1_1}
				\end{aligned}
			\end{equation}
			and
			\begin{equation}
				\begin{aligned}\sup_{t\geq 0}||n_2^p||_{L^2}^2
					\leq \max\Big\{6C^2(T_1^4+E_{0}^4+1)\chi^4p^4\sup_{t\geq0}||n_2^p||^2_{L_1}, 2||n_{2,\rm in}^p||_{L^2}^2,1 \Big\}.\nonumber
				\end{aligned}
			\end{equation}
			
			Next the Moser-Alikakos iteration is used to determine $E_{0,2}$.
			We rewrite (\ref{n1_1}) into
			\begin{equation}
				\begin{aligned}
					\sup_{t\geq 0}\int_{\mathbb{T}\times\mathbb{I}}|n_1(t)|^{2^{j+1}}dxdy\leq \max\Big\{6C^2(T_1^4+E_{0}^4+1)\chi^416^j\sup_{t\geq0}||n_1^p||^2_{L_1}, 2\int_{\mathbb{T}\times\mathbb{I}}|n_{1}(0)|^{2^{j+1}}dxdy,1 \Big\}.\label{n1_2}
				\end{aligned}
			\end{equation}
			Due to $$\sup_{t\geq0}||n_1(t)||_{L^2}\leq|T| ||n_{1,0}||_{L^\infty L^2}+||n_{1,\neq}||_{L^\infty L^2}\leq|\mathbb{T}|T_1+E_0,$$
			by interpolation, we have
			$$||n_{1}(0)||_{L^{2^j}}\leq||n_{1}(0)||^{\theta}_{L^2}
			||n_{1,\rm in}||^{1-\theta}_{L^\infty}
			\leq||n_{1}(0)||_{L^2}+||n_{1,\rm in}||_{L^\infty}\leq|\mathbb{T}|T_1+E_0+||n_{1,\rm in}||_{L^\infty},\ {\rm for}\ j\geq1,$$
			which yields
			$$2\int_{\mathbb{T}\times\mathbb{I}}|n_{1}(0)|^{2^{j+1}}dxdy
			\leq2\Big(\mathbb{T}|T_1+E_0+||n_{1,\rm in}||_{L^\infty}\Big)^{2^{j+1}}\leq K^{2^{j+1}},$$
			where $K=2(|\mathbb{T}|T_1+E_0+||n_{1,\rm in}||_{L^\infty}+1).$
			
			We infer from (\ref{n1_2}) that
			\begin{equation}
				\begin{aligned}
					\sup_{t\geq0}\int_{\mathbb{T}\times\mathbb{I}}|n_1(t)|^{2^{j+1}}dxdy\leq \max\Big\{C_1	16^{j}\Big(\sup_{t\geq0}\int_{\mathbb{T}\times\mathbb{I}}|n_1(t)|^{2^{j}}dxdy\Big)^2, K^{2^{j+1}} \Big\},\nonumber
				\end{aligned}
			\end{equation}
			where $C_1=6C^2(T_1^4+E_{0}^4+1)\chi^4.$
			
			When $j=1$, we have
			\begin{equation}
				\begin{aligned}
					\sup_{t\geq0}\int_{\mathbb{T}\times\mathbb{I}}|n_1(t)|^{2^{2}}dxdy\leq C_1^{a_1}	16^{b_1}K^{2^{2}},\nonumber
				\end{aligned}
			\end{equation}
			where $a_1=1$ and $b_1=1$.
			
			When $j=2$, we have
			\begin{equation}
				\begin{aligned}
					\sup_{t\geq0}\int_{\mathbb{T}\times\mathbb{I}}|n_1(t)|^{2^{3}}dxdy\leq C_1^{a_2}	16^{b_2}K^{2^{3}},\nonumber
				\end{aligned}
			\end{equation}
			where $a_2=1+2a_1$ and $b_2=2+2b_1$.
			
			When $j=k$, we have
			\begin{equation}
				\begin{aligned}
					\sup_{t\geq0}\int_{\mathbb{T}\times\mathbb{I}}|n_1(t)|^{2^{k+1}}dxdy\leq C_1^{a_k}	16^{b_k}K^{2^{k+1}},\nonumber
				\end{aligned}
			\end{equation}
			where $a_k=1+2a_{k-1}$ and $b_k=k+2b_{k-1}$.
			
			Generally, one can obtain the following formulas
			$$a_k=2^k-1,\ {\rm and}\ \ b_k=2^{k+1}-k-2.$$
			
			Therefore,  one obtains
			\begin{equation}
				\begin{aligned}
					\sup_{t\geq0}(\int_{\mathbb{T}\times\mathbb{I}}|n_1(t)|^{2^{k+1}}dxdy)^{\frac{1}{2^{k+1}}}\leq C_1^{\frac{2^k-1}{2^{k+1}}}	16^{\frac{2^{k+1}-k-2}{2^{k+1}}}K.\nonumber
				\end{aligned}
			\end{equation}
			Letting $k\rightarrow\infty$, there holds
			$$\sup_{t\geq0}\|n_1(t)\|_{L^\infty}\leq 32\sqrt{6}C(T_1^4+E_{0}^4+1)^{\frac{1}{2}}\chi^2(|\mathbb{T}|T_1+E_0+||n_{1,\rm in}||_{L^\infty}+1).$$
			Similarly, we have
			$$\sup_{t\geq0}\|n_2(t)\|_{L^\infty}\leq 32\sqrt{6}C(T_1^4+E_{0}^4+1)^{\frac{1}{2}}\chi^2(|\mathbb{T}|T_1+E_0+||n_{2,\rm in}||_{L^\infty}+1).$$
			
			We conclude that
			$$\|n_1\|_{L^{\infty}L^{\infty}}+\|n_2\|_{L^{\infty}L^{\infty}}\leq 32\sqrt{6}C(T_1^4+E_{0}^4+1)^{\frac{1}{2}}\chi^2(2|\mathbb{T}|T_1+2E_0+||n_{1,\rm in}||_{L^\infty}+||n_{2,\rm in}||_{L^\infty}+2):=E_{0,2}.$$
			The proof is complete.
			
		\end{proof}
		
		Combining \textbf{Lemma \ref{result_0}} with \textbf{Lemma \ref{result_1}}, setting $A_2=\max\{A_0,A_1\}$ and $$E_{0}=E_{0,1}+E_{0,2},$$ the proof of \textbf{Proposition \ref{pro1}} is complete.
		
		\section*{Acknowledgement}
		
		The authors would like to thank Professors Zhifei Zhang and Zhaoyin Xiang for some helpful communications. W. Wang was supported by NSFC under grant 12071054 and by Dalian High-level Talent Innovation Project (Grant 2020RD09).

		\section*{Declaration of competing interest}
		The authors declare that they have no known competing financial interests
		or personal relationships that could have appeared to
		influence the work reported in this paper.
		\section*{Data availability}
		No data was used in this paper.


\begin{thebibliography}{99}	
			\bibitem{Bedrossian1}Bedrossian J. and He S. (2017). Suppression of blow-up in Patlak--Keller--Segel via shear flows.
			SIAM Journal on Mathematical Analysis, 49(6), 4722-4766.
			\bibitem{Blanchet1}Blanchet A., Carrillo J.A. and Masmoudi N. (2008).
			Infinite time aggregation for the critical Patlak-Keller-Segel model in $\mathbb{R}^2$.
			Communications on Pure and Applied Mathematics, 61(10), 1449-1481.
			\bibitem{Calvez1}Calvez V. and Corrias L. (2008).
			The parabolic-parabolic Keller-Segel model in $\mathbb{R}^2$. Communications in Mathematical Sciences, 6(2), 417-447.
			\bibitem{chen} Chen Q., Li T., Wei D. and Zhang Z., (2020). Transition threshold for the 2-D Couette flow in a finite channel, Arch. Ration. Mech. Anal. 238 125-183.
			\bibitem{Coti1}Coti Zelati M., Elgindi T.M. and Widmayer K. (2020).
			Enhanced dissipation in the Navier-Stokes equations near the Poiseuille flow. Communications in Mathematical Physics, 378(2), 987-1010.
			\bibitem{Dolbeault1}Dolbeault J. and Perthame B. (2004).
			Optimal critical mass in the two dimensional Keller-Segel model in $\mathbb{R}^2$.
			Comptes Rendus Mathematique, 339(9), 611-616.
			\bibitem{ding}Ding S. and Lin Z. (2022). Enhanced dissipation and transition threshold
			for the 2-D plane Poiseuille flow via resolvent estimate. Journal of Differential Equations, 332, 404-439.
			\bibitem{Espejo1}Espejo E., Vilches K. and Conca C. (2013). Sharp condition for blow-up and global existence in a two species chemotactic keller-segel system in $\mathbb{R}^2$. European Journal of Applied Mathematics, 24(2), 297-313.
			\bibitem{Feng1}Feng Y., Shi B. and Wang W. (2022). Dissipation enhancement of planar helical flows and applications to three-dimensional Kuramoto-Sivashinsky and Keller-Segel equations. Journal of Differential Equations, 313, 420-449.
			
			\bibitem{he0}He S. (2018). Suppression of blow-up in parabolic-parabolic Patlak-Keller-Segel via strictly monotone shear flows. Nonlinearity, 31(8), 3651.
			\bibitem{he05}He S. (2023). Enhanced dissipation and blow-up suppression in a chemotaxis-fluid system. SIAM Journal on Mathematical Analysis, 55(4), 2615-2643.
			\bibitem{he1}He S. and Tadmor E. (2021). Multi-species patlak-keller-segel system. Indiana University Mathematics Journal 70(4), 1577-1624.
			
			\bibitem{Hu2023} Hu Z. (2023) Suppression of Chemotactic Singularity via Viscous Flow with Large Buoyancy. arXiv:2311.10003.
			\bibitem{Hu0} Hu Z., Kiselev A. and Yao Y. (2023) Suppression of chemotactic singularity by buoyancy. arXiv:2305.01036.
			\bibitem{Hu1}Hu Z. and Kiselev A. (2023). Suppression of chemotactic blowup by strong buoyancy in stokes-boussinesq flow with cold boundary. arxiv:2309.04349.
			\bibitem{Keller1}Keller E.F. and Segel L.A. (1970). Initiation of slime mold aggregation viewed as an instability.
			Journal of Theoretical Biology, 26(3), 399.
			\bibitem{Kiselev1}Kiselev A. and Xu X. (2016).
			Suppression of chemotactic explosion by mixing. Archive for Rational Mechanics and Analysis, 222, 1077-1112.
			\bibitem{Li0}Li H., Xiang Z. and Xu X. (2023) Suppression of chemotactic singularity
			in Patlak-Keller-Segel-Navier-Stokes system
			via the Poiseuille flow (Preprint)
			\bibitem{Li1}Li T., Wei D. and Zhang Z. (2020). Pseudospectral bound and transition
			threshold for the 3D Kolmogorov flow. Communications on Pure and Applied Mathematics, 73(3), 465-557.
			\bibitem{Patlak1}Patlak C.S. (1953). Random walk with persistence and external bias.
			The bulletin of mathematical biophysics, 15, 311-338.
			\bibitem{Schweyer1}Schweyer R. (2014).
			Stable blow-up dynamic for the parabolic-parabolic Patlak-Keller-Segel model.
			arxiv:1403.4975.
			\bibitem{wei11}Wei D. (2018). Global well-posedness and blow-up for the 2-D Patlak-Keller-Segel equation.
			Journal of Functional Analysis, 274(2), 388-401.
			\bibitem{wei}Wei D. and Zhang Z. (2021). Transition threshold for the 3D Couette flow
			in Sobolev space. Communications on Pure and Applied Mathematics, 74(11), 2398-2479.
			\bibitem{wei1}Wei D. (2021). Diffusion and mixing in fluid flow via the resolvent estimate.
			Science China Mathematics, 64, 507-518.
			\bibitem{zeng}Zeng L., Zhang Z.  and Zi R. (2021). Suppression of blow-up in
			patlak-keller-segel-navier-stokes system via the couette flow. Journal of Functional Analysis, 280(10), 108967.
		\end{thebibliography}
 \end{document}